\documentclass[11pt]{amsart}
\usepackage{latexsym,amssymb,calc,amsmath,amsthm,amssymb,color,graphicx,graphics}
\usepackage{bbold, epsfig, euscript,enumitem}

\usepackage[T1]{fontenc}
\usepackage{fancyhdr}

\usepackage{stmaryrd}
\newtheorem{theo}{Theorem}
\newtheorem{prop}{Proposition}
\newtheorem{lem}{Lemma}
\newtheorem{cor}{Corollary}
\newtheorem{defi}{Definition}

\renewcommand{\div}{{\rm div \,}}

\newcommand{\sgn}{\,\mbox{sgn}\,}

\newcommand{\esssup}[1]{\mathop{\rm ess\ sup}}
\newcommand{\essinf}[1]{\mathop{\rm ess\ inf}}
\newcommand{\N}{{\rm I\kern - 2.5pt N}}
\newcommand{\Z}{{\rm Z\kern - 5.5pt Z}}
\newcommand{\Q}{{\rm I\kern - 5.25pt Q}}
\newcommand{\C}{{\rm I\kern - 6.25pt C}}
\newcommand{\R}{{\rm I\kern - 2.5pt R}}

\begin{document}
\title[Wellposedness of the discontinuous ODE for two-phase flows]{Wellposedness of the discontinuous ODE associated with two-phase flows}
\author[Dieter Bothe]{Dieter Bothe\vspace{0.1in}\\
{\it In memoriam Jan Pr\"u\ss}
\vspace{0.1in}}
\address{Department of Mathematics and Profile Area Thermofluids \&  Interfaces\\
Technical University of Darmstadt\\
Alarich-Weiss-Str.~10\\
D-64287 Darmstadt, Germany}
\email{bothe@mma.tu-darmstadt.de}
\date{May 11, 2019}
\maketitle
$\mbox{ }$\\[-10ex]
\begin{abstract}
We consider the initial value problem
\[
\dot x (t) = v(t,x(t)) \;\mbox{ for } t\in (a,b), \;\; x(t_0)=x_0
\]
which determines the pathlines of a two-phase flow, i.e.\ $v=v(t,x)$ is a given velocity field of the type
\[
v(t,x)= \begin{cases}
v^+(t,x) &\text{ if } x \in \Omega^+(t)\\
v^-(t,x) &\text{ if } x \in \Omega^-(t)
\end{cases}
\]
with $\Omega^\pm (t)$ denoting the bulk phases of the two-phase fluid system under consideration.
The bulk phases are separated by a moving and deforming interface $\Sigma (t)$.
Since we allow for flows with phase change, these pathlines are allowed to cross or touch the interface.
Imposing a kind of transversality condition at $\Sigma (t)$, which is intimately related to the mass balance in such systems, we show existence and uniqueness of absolutely continuous solutions of the above ODE in case the one-sided velocity fields
$v^\pm:\overline{{\rm gr}(\Omega^\pm)}\to \R^n$
are continuous in $(t,x)$ and locally Lipschitz continuous in $x$.
Note that this is a necessary prerequisite for the existence of well-defined co-moving control volumes for two-phase flows,
a basic concept for mathematical modeling of two-phase continua.
\end{abstract}
\section[Introduction]{Introduction}
Given  an open interval $J=(a,b)$ in $\R$, an open set $\Omega \subset \R^n$ and $f:J\times \Omega \to \R^n$, we consider the
initial value problem
\begin{equation}\label{IVP}
\dot x (t) = f(t,x(t)) \;\mbox{ for } t\in J, \;\; x(t_0)=x_0
\end{equation}
for $t_0\in J$ and $x_0 \in \Omega$. By the classical result of Peano \cite{Peano}, problem \eqref{IVP} has a local $C^1$-solution
if $f$ is continuous. If $f$ is discontinuous in $t$, solutions will typically not be $C^1$, but absolutely continuous (a.c.\ for short) such that
\begin{equation}
x(t) = x_0 + \int_{t_0}^t f(s, x(s)) \,ds \; \mbox{ for all } t\in J.
\end{equation}
We call such a function $x (\cdot)$ an a.c.\ solution and,
again by a classical result named after C.\ Carath\'eodory,
existence of local solutions still holds true if $f$ is Lebesgue measurable in $t$ and continuous in $x$ with local integrable bounds, say
$|f(t,x)|\leq k(t)$ on $J \times \Omega$ with some $k \in L^1 (J)$; see for instance \cite{Hartman} for a proof. The solution is also called a Carath\'eodory solution of \eqref{IVP}.

The situation is more involved if $f$ is discontinuous in $x$, as it happens if $f$ denotes the velocity field in a two-phase
flow, i.e.\ in the case considered in the present paper.
More generally, discontinuous ODEs appear in several situations and possible applications which lead to such cases can be found
in \cite{Fil, MDE, BoWi, Cortes, Georgescu} and the references given there.
Here, already the simple one-dimensional example of
\[
f(x)=\alpha \sgn(x) \; \mbox{ with }\alpha \in \{-1,1\}
\]
and the sign-function $\sgn(\cdot)$ shows that (\ref{IVP}) may have no solution, a single solution or infinitely many ones, depending also on an appropriate (re-)de\-fi\-ni\-tion of $\sgn(0)$.

One way to proceed in this case is to define the multivalued (=set-valued) regularization $F:J \times \Omega \to 2^{\R^n}\setminus\{\emptyset\}$ of $f$ according to
\begin{equation}\label{E15}
F(t,x):= \bigcap\limits_{\delta>0} \overline{{\rm conv}}\, f\big(t,B_\delta(x)\cap \Omega \big) \quad\text{ for } t \in J,\; x \in \Omega
\end{equation}
and to consider the differential inclusion
\begin{equation}\label{MVIVP}
\dot x \in F(t,x(t))  \;\mbox{ for } t\in J, \;\; x(t_0)=x_0
\end{equation}
instead of \eqref{IVP}.
It is well known (see \cite{MDE}) that, given any locally bounded, measurable function $f$, the map $F$ has the following properties:
$\, F(\cdot,x)$ has a measurable selection
for every $x \in \Omega$, $F(t,\cdot)$ is upper semicontinuous (usc, for short) and $F$ is locally bounded with closed bounded convex values. Due to Theorem~5.2 in \cite{MDE},
 this is sufficient for the local-in-time existence of a.c.\ solutions of the differential inclusion \eqref{MVIVP}
for every $t_0\in J$ and $x_0 \in \Omega$. Here an a.c.\ solution is an absolutely continuous function with $x(t_0)=x_0$
and such that the inclusion in \eqref{MVIVP} holds a.e.\ on $J$.
If $F$ stems from a discontinuous function $f$ via (\ref{E15}), an a.c.\ solution of (\ref{MVIVP}) is also called a
Krasovskii solution of the discontinuous ODE \eqref{IVP}.

A variant of the above concept was introduced by A.F.\ Filippov, considering the more restrictive regularization
\begin{equation}\label{E16}
F(t,x):= \bigcap\limits_{\delta>0} \bigcap\limits_{\mu (N)=0}
\overline{{\rm conv}}\, f\big(t,B_\delta(x)\cap (\Omega \setminus N)\big) \quad\text{ for } t \in J,\; x \in \Omega.
\end{equation}
In this case, an a.c.\ solution of \eqref{MVIVP} is called a Filippov solution of \eqref{IVP} and Theorem~8 in \S 7 of \cite{Fil}
assures that a (local) a.c.\ solution of \eqref{MVIVP} exists for  measurable, integrably bounded $f$.
Observe the difference between the two variants: while the multivalued regularization ${\rm Sgn}(\cdot)$ of the sign-function according to \eqref{E16} has ${\rm Sgn}(0)=[-1,1]$, independently of the definition of ${\rm sgn }(0)$,
one always has ${\rm sgn }(0)\in {\rm Sgn}(0)$ if the latter is defined via \eqref{E15}.

While this approach immediately yields a non-empty set of solutions to the differential inclusion, these are, in general, not solutions to the original (single-valued) ODE. Employing the concept of directional continuity, A.\ Bressan obtained in \cite{Bressan88}
existence of solutions for the original initial value problem \eqref{IVP} if $|f(t,x)|\leq c$ on $J\times \Omega$ and $f$ is continuous along the cone
\[
K_\alpha := \{ (t,x)\in \R^{n+1}: |x|\leq \alpha \, t \}  \quad \mbox{ for some } \alpha >c;
\]
see \cite{Bressan88} for more details on this concept.

Another desired property is uniqueness of solutions, especially in cases where the physics of the problem asks for
single solutions to the initial value problem such as the two-phase flow problem considered below.
Local Lip\-schitz continuity of $f$ in $x$ is of course sufficient for local existence of a unique solution to \eqref{IVP} due to
the classical Picard-Lindel\"of theorem (cf.\ \cite{Hartman}). Here $f$ is also assumed to be jointly continuous, which can be
relaxed to mere measurability in $t$.
This gives forward and backward uniqueness, but evidently does not apply to $f$ being discontinuous in $x$.
If only forward uniqueness is requested, the weaker one-sided Lipschitz continuity, i.e.\
\begin{equation}\label{onesidedLip}
\langle f(t,x) - f(t,y), x-y \rangle \leq k(t) ||x-y||^2 \;\mbox{ for all } t\in J, x,y\in \Omega
\end{equation}
with $k\in L^1 (J)$ is sufficient.
Note that \eqref{onesidedLip} allows for discontinuous $f$, but imposes strong restrictions on possible jumps of $f$;
e.g., if $f\in C^1(\R \setminus \{0\})$
has one-sided limits $a^\pm$ at $x=0$, then $a^+ \leq a^-$ is necessary.
From \eqref{onesidedLip}, forward uniqueness follows by means of Gronwall's lemma,
since if $x, y$ are a.c.\ solutions of \eqref{IVP}, then
\begin{equation}\label{dissest}
\frac{d}{dt} \frac 1 2 ||x(t)-y(t)||^2 \leq k(t) ||x(t)-y(t)||^2 \;\mbox{ for a.e. } t\in J,
\end{equation}
which is why such $f$ is also said to be of dissipative type.

While $||\,\cdot \,||$ denotes the Euclidean norm in \eqref{dissest}, such an inequality with any other norm yields forward uniqueness as well. This leads to the notion of semi-inner products (as a Banach space substitute for the inner product in Hilbert spaces), a very useful concept also in infinite dimensional Banach spaces; cf.\ \S 10 in \cite{MDE}. So, the choice of different norms on $\R^n$ leads to more flexibility, but the applicability of such uniqueness criteria
is still limited. As a simple example, the semi-inner products on $(\R^n, |\,\cdot \,|_1 )$ with $|x|_1 := \sum_{i=1}^n |x_i|$ are given by
\[
(x,y)_- = |x|_1 \, \sum_{i=1}^n \min \big( y_i \, {\rm Sgn}\, x_i \big), \;\;
(x,y)_+ = |x|_1 \, \sum_{i=1}^n \max \big( y_i \, {\rm Sgn}\, x_i \big),
\]
where ${\rm Sgn}(r)={\rm sgn}(r)$ for $r\neq 0$ and ${\rm Sgn}(0)=[-1,1]\subset \R$. This yields
\[
\sum_{i=1}^n \min \big( (f_i (t,x)-f_i(t,y)) {\rm Sgn} (x_i-y_i)\big) \leq k(t) |x-y|_1
\]
with $k\in L^1 (J)$ as a criterion for forward uniqueness.
Let us note in passing that this is especially useful for $n=2$ in which case the two different sign functions in the sum can cancel.
Indeed, this yields uniqueness for right-hand sides of type
\[
f(t,x)=g(t,x) - l(t) (\phi (x), \phi (x))
\]
whenever $\phi:\R^2 \to \R$ is increasing in both variables, $l\in L^1(J)_+$ and the
$g(t,\cdot)$ are Lipschitz continuous in $x$ (or one-sided Lipschitz w.r.\ to $|\,\cdot \,|_1$) with constant $k(t)$,
where $k\in L^1(J)_+$.
While rather special, this has applications to irreversible chemical reactions; see \cite{Bo15}.

A different criterion for forward uniqueness was established in \cite{Bressan-uniqueness}, building on
the concept of directional continuity.
Theorem~1 in \cite{Bressan-uniqueness} guarantees the existence of a unique forward solution to \eqref{IVP} if $f$ has locally bounded
$K_\alpha$-variation; cf.\ also paragraph A1 in \cite{MDE}.

We consider \eqref{IVP} for discontinuous $f$, representing the velocity field of a two-phase flow inside a domain $\Omega$.
The latter is indicated by writing $v$ instead of $f$ from here on.
Such a flow field typically has jump discontinuities at a moving $C^2$-surface $\Sigma(t)$, separating the two different
fluids in their respective bulk phases $\Omega^\pm (t)$. Inside the phases, $v$ is continuous, allowing for a continuous extension
up to the boundary. Moreover, since we aim at forward and backward uniqueness,
we assume the continuous extensions of $v_{|\Omega^\pm (t)}$ to be locally Lipschitz continuous in $x$.
For this particular setup, we look for conditions on the behavior of $v$ at $\Sigma$ which guarantee wellposedness of \eqref{IVP}.
To the authors knowledge, this case is not appropriately covered by existing results on discontinuous ODEs.

We close this introduction by mentioning a different and ongoing approach to \eqref{IVP} for right-hand sides of low regularity.
If a passive scalar $\phi$ is advected by the flow field $v$, its time evolution is governed by the transport equation
\begin{equation}\label{transport}
\partial_t \phi + v \cdot \nabla \phi = 0, \qquad t\in J, x\in \Omega.
\end{equation}
Then $\phi(\cdot \,, x(\cdot \,;t_0,x_0))\equiv \phi(t_0,x_0)$, where $x(\cdot, t_0,x_0)$ is the solution of \eqref{IVP},
hence the method of characteristics can be applied if \eqref{IVP} is uniquely solvable backwards in time.
In their seminal paper \cite{DiPLions}, DiPerna and Lions initiated the investigation of how the intimate relation between the
ODE \eqref{IVP} and the scalar transport equation \eqref{transport} can be employed to obtain a flow map associated with
\eqref{IVP} for weakly differentiable velocity fields; see \cite{AmbCrippa} for a rather recent overview.
But this approach does not aim at providing solvability of \eqref{IVP} for every initial value; rather, results on the induced flow in the sense of a set-to-set map are obtained.

The main result of the present paper is the wellposedness (with forward and backward uniqueness)
of the ODE associated with the velocity field of a two-phase flow under physically meaningful assumptions.
The core idea is to establish an energy-type estimate like \eqref{dissest}, but with $||\cdot||^2$ replaced by
a different functional related to the jump conditions in two-phase flows.
In order to state our result and motivate the assumptions, some background on the physical model
as well as some auxiliary results on moving hypersurfaces are required.
\section{Sharp interface two-phase flow model}
Consider the continuum mechanical sharp-interface model for two-phase flows with phase change in a domain $\Omega \subset \R^n$ with bulk phases $\Omega^\pm(t)$, separated by a $C^2$-surface $\Sigma(t)$ such that $\Omega^+(t) \cup \Omega^-(t) \cup \Sigma(t)$ is a disjoint decomposition of $\Omega$. 

We assume that $\Sigma(t)$ is an embedded surface in $\R^n$ without boundary;
to avoid technical problems with moving contact lines (see \cite{FKB-2019} concerning mathematical difficulties
with moving contact line modeling), we actually restrict to closed surfaces. Then the balances of mass and momentum read
\begin{align}
	&\partial_t \rho+ \div(\rho v)=0 \quad \text{ in } \Omega\setminus \Sigma,\label{E1}\\
	&\partial_t(\rho v)+ \div(\rho v \otimes v-\mathcal{S})=\rho b \quad \text{ in } \Omega\setminus\Sigma,\label{E2}
\end{align}
	where $\rho$ is the mass density, $v$ the velocity, $\mathcal{S}$ the stress tensor and $b$ denotes body forces. At $\Sigma$, the transmission conditions
\begin{align}
	&\llbracket\rho(v-v^\Sigma)\rrbracket \cdot n_\Sigma=0 \quad \text{ on } \Sigma,\label{E3}\\
	&\llbracket\rho v \otimes (v-v^\Sigma)-\mathcal{S}\rrbracket \cdot
	n_\Sigma={\rm div}_\Sigma \, \mathcal{S}^\Sigma \quad \text{ on } \Sigma\label{E4}
\end{align}
are valid, where $v^\Sigma$ is the 
interface velocity, $n_\Sigma$ the interface normal field and $\mathcal{S}^\Sigma$ denotes the interface stress tensor. Note that in (\ref{E3}), (\ref{E4}) only the normal speed of displacement $V_\Sigma:= v^\Sigma \cdot n_\Sigma$ of $\Sigma (\cdot)$ enters;
cf.\ \eqref{VSigma} below for a purely kinematic definition of $V_\Sigma$.

The system (\ref{E1}) -- (\ref{E4}) requires several constitutive relations to arrive at a closed model, i.e.\ a system of PDEs for the unknown variables $\rho, v$; see \cite{Slattery-Interfaces} for more details.
Here, we are only interested in the flow generated by the two-phase velocity field. For this purpose we need to add an information on the tangential part, where we impose the standard no-slip condition, i.e.\ \begin{equation}\label{E5}
\llbracket P_\Sigma v\rrbracket=0 \quad \text{ on } \Sigma
\end{equation}
with the projector $P_\Sigma:= I-n_\Sigma \otimes n_\Sigma$. We also use $v_{||}$ as a shorthand notation for $P_\Sigma v$.
Above, the jump bracket $\llbracket\cdot\rrbracket$ is defined as
\begin{equation}\label{E6}
\llbracket \psi\rrbracket (t,x):=\lim\limits_{h \to 0+} \Big(\psi\big(t,x+h n_\Sigma (t,x)\big)- \psi\big(t,x-h n_\Sigma(t,x)\big)\Big)
\end{equation}
for $t\in J$, $x \in \Sigma(t)$.
Note also that we use "on $\Sigma$" to mean "for all $(t,x) \in {\rm gr}(\Sigma)$", where
\begin{equation}
{\rm gr}(\Sigma):=\{(t,x): x \in \Sigma(t), t \in J\} = \bigcup_{t \in J} \Big( \{t\} \times \Sigma(t) \Big)
\end{equation}
denotes the graph of the (multi-valued) map $\Sigma: J \subset \R \to 2^{\R^n}\setminus\{\emptyset\}$.
\section{Moving hypersurfaces and consistent velocity fields}
Motivated by the physical background, we employ the following
definition of a $\mathcal{C}^{1,2}$-family of moving hypersurfaces which can also be found in \cite{Kimura.2008}, \cite{PrSi15} and in a similar form in \cite{Giga.2006}.
Let us note that ${\rm div}_\Sigma \, \mathcal{S}^\Sigma$ in \eqref{E4} contains the term
$\kappa_\Sigma = {\rm div}_\Sigma (- n_\Sigma )$, which is $n-1$ times the mean curvature of $\Sigma$.
This explains the requirement that all $\Sigma (t)$ are $\mathcal{C}^2$-hypersurfaces in $\R^n$.

\begin{defi}
Let $J=(a,b)\subset \R$ be an open interval. A family $\{\Sigma(t)\}_{t \in J}$ with $\Sigma(t) \subset \R^n$ is called a
$\mathcal{C}^{1,2}$-\emph{family of moving hypersurfaces} if
\begin{enumerate}
\item[(i)]
each $\Sigma(t)$ is an orientable $\mathcal{C}^2$-hypersurface in $\R^n$ with unit normal field denoted as $n_\Sigma(t,\cdot)$;
\item[(ii)]
the graph $\mathcal{M}$ of $\Sigma$ is a $\mathcal{C}^1$-hypersurface in $\R \times \R^n$;
\item[(iii)]
the unit normal field is continuously differentiable on $\mathcal{M}$, i.e.\
\[ n_\Sigma \in \mathcal{C}^1(\mathcal{M}). \]
\end{enumerate}
\end{defi}
\noindent
We also need the notion of consistent velocity fields $v^\Sigma:\mathcal{M} \to \R^n$.
\begin{defi}\label{def-consistency}
Let $J=(a,b)\subset \R$ and $\{\Sigma(t)\}_{t \in J}$ a
$\mathcal{C}^{1,2}$-family of moving hypersurfaces in $\R^n$  with graph $\mathcal{M}$.
Let $v^\Sigma: \mathcal{M} \to \R^n$ be a continuous velocity field such that the
$v^\Sigma (t,\cdot)$ are locally Lipschitz continuous on $\Sigma (t)$ for all $t\in J$.
We say that $v^\Sigma$ and $\mathcal{M}$ are \emph{consistent} (or that $v^\Sigma$ is consistent to $\mathcal{M}$),
if the initial value problems
\begin{equation}\label{E7}
\dot{x}^\Sigma(t)=v^\Sigma\big(t, x^\Sigma(t)\big) \text{ on } J, \quad x^\Sigma(t_0)=x_0
\end{equation}
have unique a.c.\ solutions on $J$ (locally in time, forward and backward) for every $(t_0,x_0)\in \mathcal{M}$.
\end{defi}
Note that $v^\Sigma$ is only given on $\mathcal{M}={\rm gr} (\Sigma)$ in
Definition~\ref{def-consistency} above. Hence solvability of \eqref{E7} on $I\subset J$ implicitly includes the constraint
\begin{equation}
x^\Sigma(t) \in \Sigma(t) \text{ on } I.
\end{equation}
To characterize consistency, we employ the so-called intermediate cone to $\mathcal{M}$ (cf.\ \cite{AubinFrankowska}),
defined for $(t,x)\in \mathcal{M}$ by
\begin{equation}\label{E9}
T_{\mathcal{M}}(t,x):=\big\{(\tau ,v): \lim\limits_{h \to 0+} h^{-1} \;{\rm dist}\, \big(x+h v, \Sigma(t+h \tau)\big)=0\big\}.
\end{equation}
Elements of $T_{\mathcal{M}}(t,x)$ are, in general, \emph{subtangential} to $\mathcal{M}$.
At inner points of $\mathcal{M}$ (in the sense of inner point of a surface),
the intermediate cone reduces to the set of tangential vectors.
Now, as a direct consequence of Corollary~5.3 in \cite{MDE} or Theorem~13.2.1 in \cite{PW-ODE}
(cf.\ also \cite{Bo2} and the appendix in \cite{BPS-surfactant}), the following holds.
\begin{lem}\label{lemma-consistency}
Let $J=(a,b)\subset \R$ and $\{ \Sigma (t) \}_{t\in J}$ be a $\mathcal{C}^{1,2}$-family of moving hypersurfaces in $\R^n$ with graph
$\mathcal{M}$. Let $v^\Sigma: \mathcal{M} \to \R^n$ be a continuous velocity field
such that the $v^\Sigma (t,\cdot)$ are locally Lipschitz continuous on $\Sigma (t)$ for all $t\in J$.%
Then $v^\Sigma$ is consistent to $\mathcal{M}$ iff (if and only if) $v^\Sigma$ is tangential to $\mathcal{M}$ in the sense that
\begin{equation}\label{E8}
\big(1,v^\Sigma(t,x)\big) \in T_{\mathcal{M}} (t,x) \quad \mbox{ on }\mathcal{M}.
\end{equation}
\end{lem}
\noindent
For a $\mathcal{C}^{1,2}$-family $\{\Sigma(t)\}_{t \in J}$ of moving hypersurfaces, $V_\Sigma$ denotes
the \emph{speed of normal displacement} of $\Sigma(\cdot)$ and is defined via the relation
\begin{equation}\label{VSigma}
\lim_{h\to 0+} \frac 1 h {\rm dist} (x+h V_\Sigma (t,x) n_\Sigma (t,x) , \Sigma (t+h)) =0 \quad \mbox{ for } t\in J, x\in \Sigma(t).
\end{equation}
More precisely, $V_\Sigma$ should be named ''speed of normal forward displacement'' due to ''$h\to 0+$'' in \eqref{VSigma}.
But in all cases considered in the present paper, the speed of normal displacement will be the same in forward and in backward direction.
Let us note in passing that the definition via \eqref{VSigma} is equivalent to the common one which employs curves.
Indeed,
\[
V_\Sigma (t,x) = \langle \gamma' (t) , n_\Sigma (t, \gamma (t)) \rangle
\]
for any $\mathcal{C}^1$-curve $\gamma$ with $\gamma (t)=x$ and ${\rm gr}(\gamma)\subset \mathcal M$,
and the value does not depend on the choice of a particular curve; cf.\ Chapter~2.5 in \cite{PrSi15}.
In the literature, $V_\Sigma$ is often called normal velocity of $\Sigma (\cdot)$, but we prefer to call it the speed
of normal displacement since $V_\Sigma$ is not a velocity field.
The definition via \eqref{VSigma} clearly shows that $V_\Sigma$ is a purely kinematic quantity, determined only by the family $\{\Sigma(t)\}_{t \in J}$ of moving interfaces. Its computation is especially simple
if $\{\Sigma(t)\}_{t \in J}$ is given by a level set description, i.e.\
\begin{equation}\label{levelset}
\Sigma(t)=\{x \in \R^n: \phi(t,x)=0\}
\end{equation}
with $\phi \in \mathcal{C}^{1,2}(\mathcal N)$ for some open neighborhood $\mathcal N \subset \R \times \R^n$ of $\mathcal M$
such that $\nabla \phi \not= 0$ on $\mathcal{M}$. Then
\begin{equation}\label{E11}
V_\Sigma(t,x)=- \, \frac{\partial_t \phi(t,x)}{\|\nabla \phi(t,x)\|}\quad\text{ for } t \in J, \, x \in \Sigma(t).
\end{equation}
With this notation, the following characterization of consistency holds.
\begin{lem}\label{lemma-consistency2}
Let $J=(a,b)\subset \R$ and $\{ \Sigma (t) \}_{t\in J}$ be a $\mathcal{C}^{1,2}$-family of moving hypersurfaces in $\R^n$ with graph
$\mathcal{M}$. Let $v^\Sigma: \mathcal{M} \to \R^n$ be a continuous  velocity field
such that the $v^\Sigma (t,\cdot)$ are locally Lipschitz continuous on $\Sigma (t)$ for all $t\in J$.
Then $v^\Sigma$ is consistent to $\mathcal{M}$ iff
\begin{equation}\label{E9a}
v^\Sigma (t,x) \cdot n_\Sigma(t,x)=V_\Sigma(t,x) \text{ on }\mathcal{M}.
\end{equation}
\end{lem}
\begin{proof}
We first show that \eqref{E8} implies \eqref{E9a}.
Fix $(t_0,x_0)\in \mathcal{M}$ and let $(h_k)\subset \R$ with $0\neq h_k\to 0$ be given.
Then there are $z_k\in \R^n$ with $z_k \to 0$ such that
\[
x_k := x_0 - h_k v^\Sigma_{||} (t_0,x_0) + h_k z_k \in \Sigma (t_0),
\]
since $v^\Sigma_{||} (t_0,x_0)$ is tangent to $\Sigma (t_0)$ in $x_0$.
By \eqref{E8} and Lemma~\ref{lemma-consistency}, the solutions of \eqref{E7} starting in $x_k$ stay in $\mathcal{M}$, i.e.\
\[
x^\Sigma (t_0 + h_k ;t_0, x_k) \in \Sigma (t_0 + h_k) \quad \mbox{ for all } k \geq 1.
\]
Hence, with $v^\Sigma_n := \langle v^\Sigma , n_\Sigma \rangle n_\Sigma$, we obtain
\begin{align*}
{\rm dist}(x_0 + h_k v^\Sigma_n (t_0,x_0), \Sigma (t_0 + h_k) ) \leq \\[1ex]
||x_0 + h_k v^\Sigma_n (t_0,x_0) - x^\Sigma (t_0 + h_k ;t_0, x_k)||  \leq\\[1ex]
||x_0 + h_k v^\Sigma_n (t_0,x_0) - (x_k + h_k v^\Sigma (t_0, x_k)) || + h_k \delta_k
\end{align*}
with some $\delta_k \to 0+$. Therefore,
\begin{align*}
\frac{1}{h_k} {\rm dist}(x_0 + h_k v^\Sigma_n (t_0,x_0), \Sigma (t_0 + h_k) ) \leq \\[1ex]
||v^\Sigma_n (t_0,x_0) + v^\Sigma_{||} (t_0,x_0) - v^\Sigma (t_0, x_k) - z_k || + \delta_k \to 0
\quad \mbox{ as } k\to \infty.
\end{align*}
This shows that \eqref{E9a} holds at the arbitrarily chosen $(t_0,x_0)\in \mathcal{M}$.

Now we assume \eqref{E9a} to hold. Since $V_\Sigma n_\Sigma$ satisfies \eqref{VSigma}, the velocity field
$v^\Sigma_n := \langle v^\Sigma , n_\Sigma \rangle n_\Sigma$ is consistent to $\mathcal{M}$ due to Lemma~\ref{lemma-consistency}.
Hence, with obvious modifications, we can exchange the role of $v^\Sigma$ and $v_n^\Sigma$ in
the arguments from above to see that
\begin{align*}
\frac{1}{h_k} {\rm dist}(x_0 + h_k v^\Sigma (t_0,x_0), \Sigma (t_0 + h_k) )  \to 0
\quad \mbox{ as } k\to \infty,
\end{align*}
hence $(1,v^\Sigma (t_0,x_0))\in T_{\mathcal{M}} (t_0,x_0)$.
\end{proof}
The following result is a slight extension of Lemma~12 in \cite{FKB-2019} and provides the existence of a local level set representation of $\mathcal{M}={\rm gr} ( \Sigma )$ via a signed distance function.
\begin{lem}\label{signed-distance}
Let $J=(a,b)\subset \R$, $\{\Sigma(t)\}_{t \in J}$ be a $\mathcal{C}^{1,2}$-family of moving hypersurfaces in $\R^n$ and $(t_0,x_0)$ be an inner point of $\mathcal{M}={\rm gr} ( \Sigma )$.
Then there exists an open neighborhood $U \subset \R^{n+1}$ of $(t_0,x_0)$ and $\epsilon > 0$ such that the map
\begin{align*}
X: (\mathcal{M} \cap U) \times (-\epsilon, \epsilon) \rightarrow \R^{n+1},\quad
X(t,x,h) := (t, x + h \, n_\Sigma(t,x))\nonumber
\end{align*}
is a diffeomorphism onto its image
\begin{align*}
\mathcal{N}^\epsilon := X((\mathcal{M} \cap U) \times (-\epsilon , \epsilon)) \subset \R^{n+1} ,
\end{align*}
i.e.\ $X$ is invertible there and both $X$ and $X^{-1}$ are $\mathcal{C}^1$. The inverse function has the form
\begin{equation}\label{Xinverse}
X^{-1}(t,x) = (\pi_\Sigma (t,x),d_\Sigma (t,x))
\end{equation}
with $\mathcal{C}^1$-functions $\pi_\Sigma$ and $d_\Sigma$ on $\mathcal{N}^\epsilon$.
Moreover, $\nabla_x d_\Sigma  \in  \mathcal{C}^1 (\mathcal{N}^\epsilon; \R^n)$ and $\nabla_x d_\Sigma \neq 0$.
\end{lem}
\begin{proof}
The only point not covered by the proof to Lemma~12 in \cite{FKB-2019} is the additional regularity of $\nabla_x d_\Sigma $,
which follows by an argument taken from \cite{PrSi15}, where it is used for a fixed hypersurface:
given a fixed $t\in J$, we have
\[
x=\pi_\Sigma (t,x) + d_\Sigma(t,x) \, n_\Sigma (t, \pi_\Sigma (t, x)) \quad \mbox{ on } \Sigma (t),
\]
hence
\[
d_\Sigma(t,x) = \langle x - \pi_\Sigma (t, x), n_\Sigma (t, \pi_\Sigma (t, x)) \rangle
\]
by taking inner products with $n_\Sigma (t, \pi_\Sigma (t, x))$.
Differentiation as in the time-independent case (see \cite{PrSi15}) yields
\begin{equation}\label{grad-dS}
\nabla_x d_\Sigma =n_\Sigma (t, \pi_\Sigma (t, x)),
\end{equation}
hence the desired regularity of $\nabla_x d_\Sigma$ as well as $||\nabla_x d_\Sigma|| \equiv 1 \neq 0$.
\end{proof}
The latter result is useful to show that any $\mathcal{C}^{1,2}$-family of moving hypersurfaces has an intrinsic consistent
velocity field, allowing for unique solutions.
\begin{cor}\label{cor-natvelo}
Let $J=(a,b)\subset \R$ and $\{\Sigma(t)\}_{t \in J}$ be a $\mathcal{C}^{1,2}$-family of moving hypersurfaces in $\R^n$ with graph $\mathcal M$.
Then its speed of normal displacement $V_\Sigma$
is well-defined with $V_\Sigma \in \mathcal{C} (\mathcal{M})$,
$\nabla_\Sigma V_\Sigma \in \mathcal{C} (\mathcal{M};\R^n)$.
Furthermore, the intrinsic velocity field
\begin{equation}\label{natinterfacevelo}
w^\Sigma (t,x) := V_\Sigma (t,x) \, n_\Sigma (t,x) \quad \mbox{ for } (t,x) \in \mathcal M
\end{equation}
satisfies $w^\Sigma \in \mathcal{C} (\mathcal{M};\R^n)$, $\nabla_\Sigma w^\Sigma \in \mathcal{C} (\mathcal{M};\R^{n\times n})$
and is consistent to $\mathcal M$.
\end{cor}
\begin{proof}
Since only local properties are considered, it suffices to consider a fixed $(t_0,x_0)\in \mathcal M$ and arbitrarily small
neighborhoods (in $\mathcal M$) thereof. Locally, the $\mathcal{C}^{1,2}$-family $\{\Sigma(t)\}_{t \in J}$
of moving hypersurfaces is given as
\[
\Sigma (t) \cap B_\epsilon (x_0) = \{ x\in B_\epsilon (x_0): d_\Sigma (t,x)=0 \}
\]
with $d_\Sigma$ from \eqref{Xinverse} due to Lemma~\ref{signed-distance}.
Hence, by \eqref{E11} and \eqref{grad-dS}, the speed of normal displacement is given as
\[
V_\Sigma (t,x) = - \, \partial_t d_\Sigma(t,x)
\]
in a neighborhood of $(t_0,x_0)$ in $\mathcal M$. Evidently, $\partial_t d_\Sigma \in \mathcal{C} (\mathcal{M})$
by Lemma~\ref{signed-distance}, hence $V_\Sigma \in \mathcal{C} (\mathcal{M})$.
Since $n_\Sigma \in \mathcal{C}^1 (\mathcal{M})$ by assumption on $\{\Sigma(t)\}_{t \in J}$,
this also yields $w^\Sigma \in \mathcal{C} (\mathcal{M};\R^n)$.
To see the additional regularity, note that $\nabla_x d_\Sigma$ is ${\mathcal C}^1$ by Lemma~\ref{signed-distance},
hence the mixed second order
derivatives $\partial_t \partial_{x_k} d_\Sigma$ exist and are continuous. In this case, the order of differentiation can be exchanged
due to the Theorem of Schwarz\footnote{In the following refined version:
if $f:B_\epsilon (x_0)\subset \R^2\to \R$ is continuous with continuous first partial derivatives such that
$\partial_1 \partial_2 f(x)$ exists in $B_\epsilon (x_0)$ and is continuous in $x_0$, then $\partial_2 \partial_1 f(x_0)$
exists and $\partial_1 \partial_2 f(x_0)=\partial_2 \partial_1 f(x_0)$; see section~3.3 in \cite{Walter-Ana2} for a proof.}, thus $\nabla_x  \partial_t d_\Sigma$ exists and is continuous on $\mathcal{N}^\epsilon$.
Hence
\[
\nabla_\Sigma V_\Sigma = - P_\Sigma \nabla_x  \partial_t d_\Sigma \in \mathcal{C} (\mathcal{M};\R^n).
\]
Consequently,
\[
\nabla_\Sigma  w^\Sigma = n_\Sigma \otimes \nabla_\Sigma  V_\Sigma +  V_\Sigma \, \nabla_\Sigma n_\Sigma
\in \mathcal{C} (\mathcal{M};\R^{n\times n}).
\]
Finally, by definition of $V_\Sigma$, the intrinsic velocity field $w^\Sigma =V_\Sigma \, n_\Sigma$ satisfies
\[
(1, w^\Sigma (t,x)) \in T_{\mathcal M} (t,x) \quad \mbox{ on } \mathcal M.
\]
Hence $w^\Sigma$ is consistent to $\mathcal M$ due to Lemma~\ref{lemma-consistency};
note that the $w^\Sigma (t,\cdot)$ are locally Lipschitz continuous on $\Sigma (t)$ for $t\in J$.
\end{proof}
\section{Extension of consistent interface velocities}
The proof of wellposedness for the initial value problem \eqref{IVP} in the specific two-phase situation employs
a reduction to fixed $\Sigma_0$ instead of moving $\Sigma (t)$. This reduction is based on the flow map associated to \eqref{IVP}.
Recall that if the initial value problems \eqref{IVP} are wellposed, the associated \emph{flow map} (or, simply, \emph{flow})
is the map $\Phi_{t_0}^t:\R^n \to \R^n$, defined by
\begin{equation}\label{flowmap}
\Phi_{t_0}^t (x_0) := x(t;t_0,x_0),
\end{equation}
where $x(\cdot;t_0,x_0)$ is the unique solution of \eqref{IVP}. Of course, this concept can also
be defined locally if \eqref{IVP} only has local (in time) solutions. We call this the flow map associated with the right-hand side $f$.
Below, if the initial time $t_0$ is fixed, we denote the flow map as $\Phi^t$ for better readability.

Now if a $\mathcal{C}^{1,2}$-family of moving hypersurfaces in $\R^n$ is given, there is the intrinsic interface velocity field
$w^\Sigma$ given by \eqref{natinterfacevelo} and $w^\Sigma$ is consistent with the regularity as stated in Corollary~\ref{cor-natvelo}.
If $w$ denotes a continuous extension of $w^\Sigma$ from $\mathcal M:={\rm gr}(\Sigma)$ to some open neighborhood $U$ of $\mathcal M$,
being locally Lipschitz continuous in $x$, say, then the flow map $\Phi^t_{t_0}$ associated with $w$ can be used as a
nonlinear coordinate transform which fixes $\Sigma(t)$, since $\Sigma(t)=\Phi^t_{t_0} (\Sigma(t_0))$.
But this alone is not sufficient for our purpose, since a curve $\gamma (\cdot)$ which passes through $\Sigma (t_0)$
in normal direction, i.e.\ $\gamma (s_0) =: x_0 \in \Sigma (t_0)$ and (w.l.o.g.) $\gamma'(s_0) =n_{\Sigma(t_0)} ( x_0)$,
is mapped into a curve which, while crossing $\Sigma (t)$ in the point $x(t)=\Phi^t_{t_0} (x_0)$, does not pass through $\Sigma (t)$
in normal direction, in general. In other words, the coordinate transform mediated by the flow leaves the interface invariant, but
rotates the direction of vector fields, thus mixing tangential and normal parts.
To avoid this difficulty, we are going to construct a particular extension of a given consistent interface velocity field
which leads to a flow map $\Phi^t_{t_0}$ such that
\begin{equation}\label{normal-evolution}
n_{\Sigma (t)} (\Phi^t_{t_0} (y)) = \big[ D_y \Phi^t_{t_0}(y) \big]\, n_{\Sigma (t_0)} (y)
\quad \forall \, t_0 \in J, y\in \Sigma (t_0), t\in J_{t_0,y},
\end{equation}
where $J_{t_0,y}$ denotes the interval of existence of the solution to \eqref{E7} for initial value $(t_0,y)$.

A key step of this extension relies on the following
auxiliary result, where $V(r)=\omega_n |r|^n$ and $A(r)=n \omega_n |r|^{n-1}$ with $\omega_n$ the volume of
$B_1(0)\subset \R^n$.
\begin{prop}\label{prop-extension1}
Let $\Sigma$ be a $\mathcal{C}^2$-hypersurface in $\R^n$ without boundary with normal field $n$.
Due to Lemma~\ref{signed-distance}, there exists an open neighborhood $U \subset \R^n$ of $\Sigma$ such that
$\Sigma=\{ x\in U: d(x)=0\}$ with $d\in \mathcal{C}^2 (U)$ the signed distance to $\Sigma$.
Let $\pi\in \mathcal{C}^1 (U)$ denote the associated projection\footnote{Actually, $\pi$ is the metric projection onto $\Sigma$,
i.e.\  $\pi (x) \in \Sigma$ with $||x-\pi (x) ||=d(x)$.}, i.e.\ $x=\pi(x)+d(x)\, n(x)$.
Given $f^\Sigma \in \mathcal{C}^1 (\Sigma)$ and $g\in \mathcal{C} (U)$, let $\tilde U =\{ x\in U: B_{|d(x)|}(x)\subset U \}$
which is an open neighborhood of $\Sigma$.
Define $f:\tilde U\to \R$ via
\begin{equation}\label{extension1}
f(x)= f^\Sigma (\pi (x)) - \frac{d(x)}{V(d(x))} \int_{||x-y||\leq |d(x)|} g(y)\, dy \quad \mbox{ for } x\in \tilde U.
\end{equation}
Then $f$ satisfies
\begin{align}
\partial_k f (x) & = \partial_k (f^\Sigma \circ \pi )(x)\label{partialkf}
+ \partial_k d(x) \frac{n-1}{V(d(x))} \int_{||x-y||\leq |d(x)|} g(y)\, dy\\[1ex]
 & - \partial_k d(x) \frac{n}{A(d(x))} \int_{||x-y||= |d(x)|} g(y)\, do(y)\nonumber \\[1ex]
 & + \frac{n}{A(d(x))} \int_{||x-y||= |d(x)|} g(y) \frac{x_k -y_k}{d(x)}\, do(y)\nonumber
\end{align}
for all $x\in \tilde U\setminus \Sigma$, i.e.\ all $x\in \tilde U$ with $d(x)\neq 0$.
Furthermore,
\begin{equation}\label{nablafonS}
\nabla f(x) = \nabla_\Sigma f^\Sigma (x) - g(x)\, n_\Sigma (x) \quad \mbox{ for } x\in \Sigma.
\end{equation}
Finally, it holds that $f\in \mathcal{C}^1 (\tilde U)$.
\end{prop}
\begin{proof}
We consider only the case $x\in \tilde{U}^+ :=\{x\in \tilde{U}: d(x)>0 \}$, since this allows for better readability, avoiding
the use of $|d(x)|$ instead of $d(x)$; the other case can be treated by the same arguments with obvious modifications.
Evidently,
\begin{equation}\label{fG}
f(x)=(f^\Sigma \circ \pi) (x) - \frac{1}{\omega_n d(x)^{n-1}} G(x) \quad \mbox{ for } x\in \tilde{U}^+ \setminus \Sigma
\end{equation}
with
\[
G(x)=\int_{||x-y||\leq d(x)} g(y)\, dy.
\]
We have
\[
\partial_k G(x)= \frac{d}{ds} G(x+s\, e_k)_{|s=0} = \big( \frac{d}{ds} \int_{\Omega (s)} g(y)\, dy \Big)_{|s=0}
\]
and employ the Reynolds' transport theorem to compute $\partial_k G(x)$.
For this purpose note that $\Gamma (s):=\partial \Omega (s)$ has the level set representation
\[
\Gamma (s) =\{ y: \phi (s,y) = 0 \}
\quad \mbox{with } \;
\phi (s,y)=||x+s\, e_k -y||^2 - d(x+s\, e_k)^2.
\]
Using \eqref{E11}, a simple calculation shows that $\Gamma (\cdot )$ has normal speed of displacement $V_\Gamma$ given by
\[
V_\Gamma (s,y)=\frac{- \partial_s \phi(s,y)}{||\nabla_y \phi(s,y)||} =
\frac{d(x+s\, e_k)\frac{d}{ds}d(x+s\, e_k)-x_k +y_k -s}{||x-y+s\, e_k ||}.
\]
Hence
\[
\partial_k G(x)=\int_{\Gamma(0)} g(y) \frac{d(x)\partial_k d(x) +y_k -x_k}{||x-y||}\, do(y),
\]
and therefore
\begin{equation}\label{partial_kG(x)}
\partial_k G(x)=\partial_k d(x) \int_{||x-y||=d(x)} \!\! g(y) do(y)
- \! \int_{||x-y||=d(x)} \!\! g(y) \frac{x_k \!- \!y_k}{d(x)}\, do(y).
\end{equation}
Differentiating \eqref{fG}, using \eqref{partial_kG(x)}, yields \eqref{partialkf} for all $x\in \tilde{U}^+\setminus \Sigma$.

At $x\in \Sigma$ we have $f(x)=f^\Sigma (\pi (x))=f^\Sigma (x)$. Hence, for $s>0$,
\[
f(x+s\, n) = f(x) - \frac{s}{V(s)} \int_{||x+s\hspace{1pt} n-y||\leq s} g(y)\, dy
\]
with $n:=n_\Sigma (x)$. Thus,
\begin{align*}
|| & \frac{f(x+s\, n) - f(x)}{s} + g(x) || \leq \frac{1}{V(s)} \int_{||x+s\hspace{1pt} n-y||\leq s} ||g(x) - g(y)||\, dy\\
 & \leq \sup \{ ||g(x) - g(y)||: ||x+s\hspace{1pt} n-y||\leq s \} \to 0 \quad \mbox{as } s\to 0+.
\end{align*}
It is easy to check (replacing $s$ by $|s|$ at a few places) that the same conclusion holds for $s\to 0-$, hence
\begin{equation}\label{fnderivative}
\frac{\partial f}{\partial n} (x) = - g(x) \quad \mbox{at } x\in \Sigma.
\end{equation}
On $\Sigma$, we also have $\nabla_\Sigma f (x) =\nabla_\Sigma f^\Sigma (x)$ since $f=f^\Sigma$ there.
Together with \eqref{fnderivative}, this yields \eqref{nablafonS}.

To finish the proof, notice first that the $\partial_k f$ are continuous on $\tilde{U}^+\setminus \Sigma$.
Indeed, there are two types of averages involved in \eqref{partialkf}, namely volume averages
\[
x \to \frac{1}{V(d(x))} \int_{||x-y||\leq d(x)} h(y)\, dy
\]
and area averages
\[
x \to \frac{1}{A(d(x))} \int_{||x-y||= d(x)} h(y)\, do(y)
\]
with functions $h\in \mathcal{C}(U)$.
The continuity of these maps follows from continuity of $h$ and $d$ by the dominated convergence theorem,
if the integrals are rewritten via rescaling as
\[
x \to \frac{1}{\omega_n} \int_{||z||\leq 1} h(x+d(x) z)\, dz
\]
and
\[
x \to \frac{n}{\omega_n} \int_{||z||= 1} h(x+d(x) z)\, do(z).
\]
It remains to show that
\begin{equation}\label{convnablaf}
\nabla f(x) \to \nabla_\Sigma f^\Sigma (x_0) - g(x_0)\, n_\Sigma (x_0) \quad \mbox{ for }
\tilde{U}^+\setminus \Sigma \ni x\to x_0 \in \Sigma.
\end{equation}
For $x\in \tilde{U}^+\setminus \Sigma$, we have
\begin{equation}\label{c1}
|| \frac{1}{V(d(x))} \int_{||x-y||\leq d(x)} g(y)\, dy - g(x) || \leq \sup_{||x-y||\leq d(x)} ||g(x) - g(y)||,
\end{equation}
\begin{equation}\label{c2}
|| \frac{1}{A(d(x))} \int_{||x-y||= d(x)} g(y)\, do(y) - g(x) || \leq \sup_{||x-y||= d(x)} ||g(x) - g(y)||
\end{equation}
and
\begin{equation}\label{c3}
\int_{||x-y||= d(x)} \!\! g(y)\frac{x_k-y_k}{d(x)} do(y)
= \int_{||x-y||= d(x)} \!\!  \big( g(y)-g(x)\big) \frac{x_k \!-\! y_k}{d(x)} do(y).
\end{equation}
For the latter equality, note that
\[
\int_{||x-y||= d(x)} \frac{x_k-y_k}{d(x)}\, do(y) = \frac{1}{d(x)} \int_{||z||= d(x)} z_k \, do(z) = 0.
\]
Applying the relations \eqref{c1}, \eqref{c2} and \eqref{c3} to \eqref{partialkf} immediately yields
\eqref{convnablaf}, hence $f\in \mathcal{C}^1 (\tilde{U}^+)$. Together with the analogous treatment for $x\in \tilde{U}^-$
and because the limit on $\Sigma$ is the same for both sides, we obtain $f\in \mathcal{C}^1 (\tilde{U})$.
\end{proof}
Let us note in passing that, in vector notation, equation \eqref{partialkf} means
\begin{align}
\nabla f (x) & = \nabla (f^\Sigma \circ \pi )(x)
+ n_\Sigma(x) \frac{n-1}{V(d(x))} \int_{||x-y||\leq |d(x)|} g(y)\, dy\\[1ex]
 & - n_\Sigma (x) \frac{n}{A(d(x))} \int_{||x-y||= |d(x)|} g(y)\, do(y)\nonumber \\[1ex]
 & - \frac{n}{A(d(x))} \int_{||x-y||= |d(x)|} g(y) \nu (y)\, do(y) \quad \mbox{ for } x\in \tilde U \setminus \Sigma,\nonumber
\end{align}
where $\nu (\cdot)$ is the outer unit normal to the sphere $\partial B_{d(x)} (x)$.

Inspection of the above proof in the time-dependent case shows that
the following result is an immediate corollary to Proposition~\ref{prop-extension1}.
\begin{cor}\label{extension2}
Let $J=(a,b)\subset \R$ and $\{\Sigma(t)\}_{t \in J}$ be a $\mathcal{C}^{1,2}$-family of moving hypersurfaces without
boundary in $\R^n$ with graph $\mathcal M$. By Lemma~\ref{signed-distance},
there exists an open neighborhood $\mathcal N \subset \R^{n+1}$
of $\mathcal M$ such that $\{\Sigma(t)\}_{t \in J}$
has a level set representation with signed distance function $d_\Sigma$ such that $d_\Sigma \in \mathcal{C}^1 (\mathcal N)$
and $\nabla_x d_\Sigma \in \mathcal{C}^1 (\mathcal N;\R^n)$. Let $\pi_\Sigma \in \mathcal{C}^1 (\mathcal N)$ denote the
associated family of projections onto $\Sigma (\cdot)$ characterized by
\[
x=\pi_\Sigma (t,x) + d_\Sigma (t,x) \, n_\Sigma (t,x) \quad \mbox{ for all } (t,x)\in \mathcal N.
\]
Given $f^\Sigma \in \mathcal{C}(\mathcal M)$ with $\nabla_\Sigma f^\Sigma \in \mathcal{C}(\mathcal M)$
and $g\in \mathcal{C}(\mathcal N)$, let $\mathcal U$ with $\mathcal M \subset \mathcal U \subset \mathcal N$
be open and so small that $(t,x)\in \mathcal U$ implies $\{t\}\times B_{|d(t,x)|} (x) \subset \mathcal N$.
Define $f:\mathcal U \to \R$ by means of
\begin{equation}\label{extension-time}
f(t,x)= f^\Sigma (t,\pi (t,x)) - \frac{d(t,x)}{V(d(t,x))} \int_{||x-y||\leq |d(t,x)|} g(t,y)\, dy
\;\; \mbox{ on } 
\mathcal U
\end{equation}
with $d:=d_\Sigma$ and $\pi:=\pi_\Sigma$.
Then $f\in \mathcal{C} (\mathcal U)$ and $f(t,\cdot)\in \mathcal{C}^1 (\mathcal U^t)$, where
$\mathcal U^t := \{x\in \R^n : (t,x)\in \mathcal U \}$ is an open neighborhood of $\Sigma (t)$.
Moreover, the spatial derivatives $\partial_{x_k} f$ are given by \eqref{partialkf} on $\mathcal M$,
and by \eqref{nablafonS} on $\mathcal U \setminus \mathcal M$
with obvious modifications in form of the additional variable $t$.
\end{cor}
We are now able to prove the following key extension result.
\begin{lem}\label{part-ext}
Let $J=(a,b)\subset \R$ and $\{ \Sigma (t) \}_{t\in J}$ be a $\mathcal{C}^{1,2}$-family of moving hypersurfaces
without boundary in $\R^n$ with $\mathcal{M} = {\rm gr} (\Sigma)$.
Let $v^\Sigma \in \mathcal{C}(\mathcal{M};\R^n)$ be consistent to $\mathcal{M}$
with $\nabla_\Sigma v^\Sigma \in \mathcal{C}(\mathcal{M};\R^{n\times n})$.
Then there exists a neighborhood $\mathcal U$ of $\mathcal M$ and an extension
$\hat{v}^\Sigma :\mathcal U \to \R^n$ of $v^\Sigma$ being jointly continuous
and locally Lipschitz continuous in $x$ such that,
with $\Phi^t_{t_0}$ the (local) flow map associated to $\hat{v}^\Sigma$, the evolution of the normal field satisfies \eqref{normal-evolution}.
In particular, the intrinsic surface velocity $v^\Sigma=V_\Sigma n_\Sigma$ admits such an extension.
\end{lem}
\begin{proof}
Since the statement is about local properties of the desired extension, we may consider a small neighborhood
$U_\epsilon =(\eta-\epsilon, \eta + \epsilon) \times B_\epsilon (\xi)$ of
a point $(\eta,\xi)\in \mathcal M$ in which the moving hypersurfaces are given by means of the signed distance function
from Lemma~\ref{signed-distance}. We then extend the given function $v^\Sigma$ from $\mathcal M \cap U_\epsilon$ to a function
$\hat{v}^\Sigma$ on $U_\epsilon$ by means of
\begin{equation}\label{ext-vsigma}
\hat{v}^\Sigma (t,x) := v^\Sigma (t,\pi (t,x))- \frac{d (t,x)}{V(d(t,x))} \, F(t,x)\vspace{-0.08in}
\end{equation}
with\vspace{-0.28in}\\
\begin{equation}\label{Fdef}
F(t,x)= \int_{B_{|d(t,x)|}(x)} \sum_{k=1}^{n-1}
\langle \frac{\partial v^\Sigma}{\partial \tau_k}(t,\pi (t,y)), n_\Sigma (t,\pi (t,y)) \rangle \,
\tau_k (t,\pi (t,y)) dy,
\end{equation}
where $d:=d_\Sigma$, $\pi:=\pi_\Sigma$ is the projection from Lemma~\ref{signed-distance} and
\begin{equation}
\{ \tau_k (t,x) : k=1, \ldots , n-1 \}
\;\;  \mbox{ for } (t,x)\in \mathcal M \cap U_\epsilon
\end{equation}
is an orthonormal basis of the tangent space to $\Sigma(t)$ at the point $x$,
depending continuously differentiable on $(t,x)\in \mathcal M \cap U_\epsilon$.
Note that we obtain such an orthonormal basis with the desired regularity by applying the
Gram-Schmidt orthonormalization procedure to the system
\begin{equation}
\{\tau_k^0 - \langle \tau_k^0, n_\Sigma (t,x) \rangle \, n_\Sigma (t,x) \,:\, k = 1,\ldots ,n-1 \}
\end{equation}
with $\{\tau_k^0 :\, k = 1,\ldots ,n-1 \}$ being a basis of the tangent space to $\Sigma (\eta)$ at the point $\xi$.
By choosing $\epsilon >0$ sufficiently small, this is a system of linearly independent vectors on $\mathcal M \cap U_\epsilon$ and
the elements depend continuously differentiable on $(t,x)$ since $n_\Sigma$ has this regularity.

Now observe that the components of $\hat{v}^\Sigma (t,x)$ in \eqref{ext-vsigma} are precisely of the type as given in
\eqref{extension-time} and the integrand in \eqref{Fdef} is continuous due to our assumptions on $\{ \Sigma (t) \}_{t\in J}$
and $v^\Sigma$.
Therefore, by Corollary~\ref{extension2}, $\hat{v}^\Sigma$ is continuous in
$\mathcal M \cap U_\epsilon$ and the $\hat{v}^\Sigma (t,\cdot)$ are continuously differentiable on a neighborhood of $\Sigma (t)$.
In particular, $\hat{v}^\Sigma$ is jointly continuous and locally Lipschitz continuous in $x$ and, hence,
the initial value problems \eqref{IVP} are uniquely solvable for right-hand side $\hat{v}^\Sigma$, at least locally in time.
Consequently, the associated flow map $\Phi^t_{t_0}$ is welldefined.
Moreover, $\Phi^t_{t_0}$ is invertible with inverse $\Phi^{t_0}_t$, hence a diffeomorphism due to the regularity of $\hat{v}^\Sigma$.
Thus, $D_y \Phi^t_{t_0}(y)$ is invertible.

Moreover, by Corollary~\ref{extension2}, we also know that $\hat{v}^\Sigma$ satisfies
\begin{equation}\label{normal-derivative}
\frac{\partial \hat{v}^\Sigma}{\partial n_\Sigma}(t,x) =
- \sum_{k=1}^{n-1} \langle \frac{\partial v^\Sigma}{\partial \tau_k}(t,x), n_\Sigma (t,x) \rangle \, \tau_k (t,x)
\;\; \mbox{ for } (t,x)\in \mathcal M \cap U_\epsilon.
\end{equation}

In order to prove \eqref{normal-evolution}, we consider the equivalent relation
\begin{equation}\label{normal-evolution2}
\big[ D_y \Phi^t_{t_0}(y) \big]^{-1}  n_{\Sigma (t)} (\Phi^t_{t_0} (y)) = n_{\Sigma (t_0)} (y)
\quad \forall \, t_0 \in J, y\in \Sigma (t_0), t\in J_{t_0,y}.
\end{equation}
Evidently, equation \eqref{normal-evolution2} holds for $t=t_0$. Therefore, it holds for all $t\in J_{t_0,y}$, if
we show that the $t$-derivative of the left-hand side vanishes. We have
\begin{align*}
\frac{d}{dt} \big[ D_y \Phi^t_{t_0}(y) \big]^{-1}  n_{\Sigma (t)} (\Phi^t_{t_0} (y)) =\\[1ex]
-\, \big[ D_y \Phi^t_{t_0}(y) \big]^{-1}  \partial_t D_y \Phi^t_{t_0}(y) \big[ D_y \Phi^t_{t_0}(y) \big]^{-1} n_{\Sigma (t)} (\Phi^t_{t_0} (y))\\[1ex]
+ \big[ D_y \Phi^t_{t_0}(y) \big]^{-1} \, \frac{d}{dt} n_{\Sigma (t)} (\Phi^t_{t_0} (y)).
\end{align*}
We now employ Schwarz' theorem to get
\[
\partial_t D_y \Phi^t_{t_0}(y) = D_y \partial_t \Phi^t_{t_0}(y) = D_y \hat{v}^\Sigma (t,\Phi^t_{t_0} (y))
= \nabla_x \hat{v}^\Sigma (t,\Phi^t_{t_0} (y)) \, D_y \Phi^t_{t_0}(y)
\]
which yields
\begin{align*}
\frac{d}{dt} \big[ D_y \Phi^t_{t_0}(y) \big]^{-1}  n_{\Sigma (t)} (\Phi^t_{t_0} (y)) =\\[1ex]
\big[ D_y \Phi^t_{t_0}(y) \big]^{-1} \, \Big(
\frac{d}{dt} n_{\Sigma (t)} (\Phi^t_{t_0} (y)) - \nabla_x \hat{v}^\Sigma (t,\Phi^t_{t_0} (y)) \, n_{\Sigma (t)} (\Phi^t_{t_0} (y))
\Big).
\end{align*}
Due to Theorem~4 in \cite{FKB-2019} (extended from hypersurface in $\R^3$ to $\R^n$), the Lagrangian derivative of the normal field
satisfies
\begin{equation}
\frac{d}{dt} n_{\Sigma (t)} (\Phi^t_{t_0} (y)) =
- \sum_{k=1}^{n-1} \langle \frac{\partial v^\Sigma}{\partial \tau_k}(t,\Phi^t_{t_0} (y)), n_{\Sigma (t)} (\Phi^t_{t_0} (y)) \rangle
\, \tau_k (t,\Phi^t_{t_0} (y)).
\end{equation}
This relation, together with the normal derivative of $\hat{v}^\Sigma$ according to \eqref{normal-derivative}
shows that
\[
\frac{d}{dt} \big[ D_y \Phi^t_{t_0}(y) \big]^{-1}  n_{\Sigma (t)} (\Phi^t_{t_0} (y)) =0
\]
along the solution of \eqref{E7}, hence \eqref{normal-evolution} holds.
\end{proof}
\section{The ODE associated with a two-phase flow}
Let $\Omega \subset \R^n$ be an open set, denoting the domain of a two-phase flow.
We consider a $\mathcal{C}^{1,2}$-family of moving hypersurfaces which decomposes $\Omega$ into disjoint sets
according to $\Omega=\Omega^+(t) \cup \Omega^-(t) \cup \Sigma(t)$.
We focus on the case when the $\Sigma(t)$ are hypersurfaces of $\R^n$ without boundary. Hence $J \times \Omega$ is cut by $\mathcal{M}={\rm gr}(\Sigma)$ into two (not necessarily connected) parts $G^+$ and $G^-$, where $G^\pm={\rm gr}(\Omega^\pm)$.
Now, let $v^\pm:\overline{G^\pm}\to \R^n$ be continuous vector fields which are locally Lip\-schitz continuous
in $x$, separately on $\overline{G^+}$, respectively $\overline{G^-}$.
We also assume at most linear growth in $x$, i.e.\
\begin{equation}\label{growth-pm}
| v^\pm (t,x) | \leq c\, (1 + |x|) \;\mbox{ for all } t\in J, x\in \Omega^\pm (t)
\end{equation}
with some $c>0$.
We denote by $v$ without superscript the map with values $v^\pm$ on $\overline{G}^\pm$
which is not uniquely defined on $\mathcal{M}$, but attains two possible distinct values there, i.e.\ $v$ is multi-valued on $\mathcal{M}$. We then study the discontinuous differential equation
\begin{equation}\label{E12}
\dot{x}(t)=v\big(t,x(t)\big) \;\text{ on } J,\quad x(t_0)=x_0
\end{equation}
for $t_0 \in J$, $x_0 \in \Omega$. Note that we are slightly abusing notation here, since it should actually read
\[
\dot{x}(t)\in v\big(t,x(t)\big) \;\text{ on } J,\quad x(t_0)=x_0.
\]
But this is not relevant if, along the solution, the multivaluedness of $v$ only occurs for $t$ from a set of Lebesgue measure zero. We hence stick to \eqref{E12} and employ the following solution concept.
\begin{defi}
We call an absolutely continuous function $x:J \to \R^n$  a solution of (\ref{E12}), if $x(t_0)=x_0$,
$N:=\{t \in J: v\big(t, x(t)\big) \text{ is multivalued} \hspace{1pt} \}$
is a set of Lebesgue measure zero and $\dot{x}(t)=v\big(t, x(t)\big)$ a.e.\ on $J\setminus N$.
\end{defi}

We are interested in physically relevant conditions on $v$ and $\Sigma$ such that (\ref{E12}) has unique strong solutions, locally in time, for every initial value. Motivated by (\ref{E3}), we impose the transmission condition
\begin{equation}\label{transcond}
\rho^+ (v^+-v^\Sigma)\cdot n_\Sigma=\rho^- (v^--v^\Sigma)\cdot n_\Sigma
\quad\text{ on } \mathcal{M}
\end{equation}
with locally Lipschitz functions $\rho^\pm :\overline{G^\pm}\to (0,\infty)$.
Observe that this implies the transversality-type condition
\begin{equation}\label{E13}
\sgn_{\!0} \big((v^+-v^\Sigma)\cdot n_\Sigma\big)=\sgn_{\!0}\big((v^--v^\Sigma)\cdot n_\Sigma\big)\quad\text{ on } \mathcal{M},
\end{equation}
where $\sgn_{\!0} (0):=0$; recall that $v^\pm$ have unique one-sided limits at every $x \in \Sigma(t)$, $t \in J$.

In addition, we assume (\ref{E5}) to hold, i.e.\ the tangential parts of $v^\pm$ satisfy
\begin{equation}\label{E14}
v^+_{\|}=v^-_{\|}\quad\text{ on } \mathcal{M}.
\end{equation}
Since $v^\Sigma$ enters our assumptions only via $V_\Sigma=v^\Sigma \cdot n_\Sigma$, we may assume
\begin{equation}\label{vtan}
v^\Sigma_{\|}=v^\pm_{\|}.
\end{equation}
\section{Wellposedness of the ODE from two-phase flow}
We now give the main result of this paper.
\begin{theo}\label{T1}
Let $J=(a,b)\subset \R$ and $\{ \Sigma (t) \}_{t\in J}$ be a $\mathcal{C}^{1,2}$-family of moving hypersurfaces
in $\R^n$ without boundary which divide an open set $\Omega\subset \R^n$ into $\Omega^+(t)\cup \Omega^- (t) \cup \Sigma(t)$
for all $t\in J$ with time-dependent bulk phases $\Omega^\pm (t)$.
Let\\[-2ex]
\[
	v^\pm: {\rm gr}\big(\overline{\Omega^\pm(\cdot)}\big) \to \R^n
\]
be continuous in $(t,x)$ and locally Lipschitz continuous in $x$ such that (\ref{transcond}) and (\ref{E14}) are valid,
where $v^\Sigma :=V_\Sigma n_\Sigma$ is the consistent intrinsic interface velocity associated to $\{ \Sigma (t) \}_{t\in J}$.
Then, for given $t_0 \in J$ and $x_0 \in \Omega$, the initial value problem (\ref{E12}) has a unique a.c.\ solution, locally in time.
This solution is also the unique Filippov solution of (\ref{E12}).
\end{theo}
\noindent
{\bf Proof.}
The proof is given in several steps.\\[0.5ex]
{\bf Step 1.} \emph{Existence of solutions.}\\
In the specific situation under consideration, one can easily see that $F$ from \eqref{E15} is given by
\begin{equation}\label{E17}
F(t,x)= \begin{cases}
\{v^+(t,x)\} &\text{ if } x \in \Omega^+(t),\\
{\rm conv}\{v^+(t,x), v^-(t,x)\} &\text{ if } x \in \Sigma(t),\\
\{v^-(t,x)\} &\text{ if } x \in \Omega^-(t).
\end{cases}
\end{equation}
This multivalued map is even jointly usc such that classical existence results for differential inclusions with usc right-hand side apply; see \cite{Ac}, \cite{MDE}. Therefore, concerning the existence part,
it only remains to show that any a.c.\ solution $x(\cdot)$ of \eqref{MVIVP} with $F$ from (\ref{E17}) is actually an a.c.\ solution of (\ref{E12}). For this purpose, we will show that
\[
M:=\{t\in J: F(t,x(t)) \mbox{ is multivalued} \hspace{1pt} \}
\]
is a Lebesgue null set.
Evidently, $M\subset N:=\{t \in J: x(t) \in \Sigma(t)\}$, since for $t \in J\setminus N$ it holds that $F(t,x)=\{v(t,x)\}$, hence $\dot{x}(t)=v\big(t,x(t)\big)$ a.e.\ on $J\setminus N$. Since $x(\cdot)$ is a.c., the derivative $\dot{x}(t)$ exists a.e.\ on $J$, in particular a.e.\ on $N$. Given a (local)
level set representation of $\Sigma$ according to \eqref{levelset}, we have
\[
\phi\big(t,x(t)\big)=0 \quad\text{ on } N,
\]
hence also
\[
0= \frac{{\rm d}}{{\rm d}t} \phi \big(t,x(t)\big)= \partial_t \phi \big(t,x(t)\big)+ \dot{x}(t) \cdot \nabla \phi \big(t,x(t)\big) \quad \text{ a.e.\ on } N.
\]
Note that such a level set representation exists at least locally due to our regularity assumptions on $\Sigma$
by Lemma~\ref{signed-distance}.
Using \eqref{E11}, this implies
\[
\dot{x}(t) \cdot n_\Sigma \big(t,x(t)\big)= V_\Sigma \big(t,x(t)\big) \quad \text{ a.e. on } N.
\]
On the other hand,
\[
P_\Sigma \dot{x}(t) \in P_\Sigma F(t,x(t)) = \{ v^\pm_{||} (t,x(t)) \}
\]
due to \eqref{E14}.
Therefore, employing \eqref{vtan}, we obtain
\begin{equation}\label{E18}
\dot{x}(t)= V_\Sigma \big(t,x(t)\big) n_\Sigma\big(t,x(t)\big)+ v^\pm_{\|}\big(t,x(t)\big)
= v^\Sigma \big(t,x(t)\big) \quad\text{ a.e.\ on } N.
\end{equation}
Consequently,
\[
v^\Sigma \big(t,x(t)\big) \in {\rm conv}\{v^+(t,x(t)), v^-(t,x(t))\}
\quad\text{ for all } t\in N_0,
\]
where $N_0\subset N$ has $\lambda_1 (N\setminus N_0)=0$.
Taking inner product with $n_\Sigma$, this implies (with a slight abuse of notation)
\[
0 \, \in \, \big( {\rm conv}\{(v^+-v^\Sigma ) \cdot n_\Sigma ,
(v^- -v^\Sigma )\cdot n_\Sigma \} \big) \big( t,x(t) \big)
\quad\text{ for all } t\in N_0.
\]
For fixed $t\in N_0$, two cases are hence possible: either
\begin{equation}\label{doubleinequ}
\big(v^+-v^\Sigma \big) \cdot n_\Sigma \, \leq \, 0 \,  \leq \, \big(v^- -v^\Sigma \big)\cdot n_\Sigma
\quad \mbox{ at } (t,x(t))
\end{equation}
or the same with $v^+$, $v^-$ exchanged. We only consider the first case and assume that strict inequality holds
at least for one relation in \eqref{doubleinequ}. Then,
after multiplication by the factors $\rho^\pm >0$ on the respective side, we obtain
\[
\rho^+ \big(v^+-v^\Sigma \big) \cdot n_\Sigma < \rho^- \big(v^- -v^\Sigma \big)\cdot n_\Sigma
\quad \mbox{ at } (t,x(t)),
\]
a contradiction to the transversality condition (\ref{E13}).
This shows that
\[
v^+ \cdot n_\Sigma = v^\Sigma \cdot n_\Sigma  = v^- \cdot n_\Sigma \quad \mbox{ at } (t,x(t)).
\]
To sum up, it therefore holds that
\[
v^+(t,x(t))=v^-(t,x(t))\quad\text{ for all } t\in N_0,
\]
hence
$t\in N_0$ implies $t\not \in M$, i.e.\ $M\subset N\setminus N_0$ and thus $M$ is a null set.
Note that, up to here, less regularity of $v^\pm$ would be sufficient, say measurability in $t$ and local Lipschitz
continuity in $x$.\\

It remains to show \emph{uniqueness of a.c.\ solutions,} where we start with forward uniqueness.
For this purpose, let $x(\cdot)$ and $\overline{x}(\cdot)$ be two (distinct) strong solutions of (\ref{E12}) with common initial value $x_0$. Local-in-time (forward and backward)
uniqueness is clear in case $x_0 \notin \Sigma(t_0)$. So, we may assume $x_0 \in \Sigma(t_0)$ and have to show that $x(t)= \overline{x}(t)$ on $[t_0,t_0+\delta]$ for some $\delta>0$.
\\[1ex]
{\bf Step 2.} \emph{Reduction to fixed $\Sigma$ and vanishing tangential part of $v$.}\\[0.5ex]
Let $\hat{v}^\Sigma$ be the extension of $v^\Sigma:\mathcal{M}\to \R^n$ provided by Lemma~\ref{part-ext}.
Considering only local wellposedness, we may assume that $v^\Sigma$ and then also $\hat{v}^\Sigma$ are bounded
and that we may assume that $\hat{v}^\Sigma$ is given on all of $J \times \R^n$.
Hence $\hat{v}^\Sigma$ generates a global flow $\Phi^t_{t_0}:\R^n \to \R^n$ via $\Phi^t_{t_0}(y_0):=y(t;t_0,y_0)$, where $y(\cdot, t_0,y_0)$ is
the unique global solution of
\begin{equation}\label{E19}
\dot{y}(t)= \hat{v}^\Sigma\big(t,y(t)\big) \quad\text{ on } J,\quad y(t_0)=y_0.
\end{equation}
Note that $\Phi$ leaves $\Sigma(\cdot)$ invariant, which means that
\begin{equation}\label{E20}
\Sigma(t)= \Phi^t_{t_0} \big(\Sigma(t_0)\big) \quad\text{ for all } t, t_0 \in J.
\end{equation}
This follows by Lemma~\ref{lemma-consistency2}, since the vector field $v^\Sigma$ is consistent to $\mathcal M$.
Moreover, $\Phi^t_{t_0}$ also leaves $\Omega^+(\cdot)$, respectively $\Omega^-(\cdot)$ invariant since solutions cannot cross $\Sigma(\cdot)$ due to unique solvability.
Now $x(\cdot)$ is a strong solution of (\ref{E12}) iff the a.c.\ function $y(\cdot)$, implicitly defined by
\begin{equation}\label{E21}
x(t)= \Phi^t_{t_0} (y(t)),
\end{equation}
solves the initial value problem
\begin{equation}\label{E22}
\dot y (t)=f\big(t,y(t)\big)\quad\text{ on } J, \quad y(t_0)=x_0
\end{equation}
with right-hand side $f$ given by
\begin{equation}\label{E23}
f(t,y):= \big[ D_y \Phi^t_{t_0}(y)\big]^{-1}  \cdot \big( v (t, \Phi^t_{t_0}(y))- \hat{v}^\Sigma (t, \Phi^t_{t_0}(y)) \big).
\end{equation}
Note that $f$ is discontinuous at $(\Phi^t_{t_0})^{-1} \big(\Sigma(t)\big)=\Sigma(t_0)=:\Sigma_0$, but
the $f^\pm$ given by the right-hand side of \eqref{E23} on $\Omega^\pm(t_0)$
have the same regularity as the $v^\pm$, with continuous extensions onto the closure of $\Omega^\pm(t_0)$.
To rewrite the transmission condition (\ref{transcond}), let
\[
\hat{\rho}^\pm (t,y):= \rho^\pm (t, \Phi^t_{t_0}(y)), \qquad n(y):=n_{\Sigma_0} (y)
\]
and note that the $\hat{\rho}^\pm$ have the same regularity as the $\rho^\pm$. Then, for $y\in \Sigma_0$,
\[
\hat{\rho}^+ (t,y) f^+ (t,y) \cdot n(y) = \rho^+ (t,x) \big[ D_y \Phi^t_{t_0}(y)\big]^{-1}
(v^+ (t,x)- \hat{v}^\Sigma (t,x))\cdot n(y)
\]
with $x=\Phi^t_{t_0}(y)\in \Sigma(t)$. Due to \eqref{vtan}, we have
\[
v^+ (t,x)- \hat{v}^\Sigma (t,x) = \langle v^+ (t,x)- \hat{v}^\Sigma (t,x), n_\Sigma (t,x)\rangle \, n_\Sigma (t,x),
\]
hence (with shorthand notation)
\[
\big(\hat{\rho}^+ f^+ \big)(t,y) \cdot n(y) = \big( \rho^+  \langle v^+ - \hat{v}^\Sigma , n_\Sigma \rangle \big)(t,x)
\big[ D_y \Phi^t_{t_0}(y)\big]^{-1} n_\Sigma (t,x) \cdot n(y).
\]
Rewriting $\big( \hat{\rho}^- f^+ \big) (t,y) \cdot n(y)$ in an analogous way, we see that
(\ref{transcond}) becomes
\begin{equation}\label{E24a}
\hat{\rho}^+ f^+ \cdot n=\hat{\rho}^- f^- \cdot n \quad\text{ on } \Sigma_0,
\end{equation}
and the transversality condition (\ref{E13}) becomes
\begin{equation}\label{E24}
\sgn_{\!0} (f^+ \cdot n)=\sgn_{\!0} (f^- \cdot n)\quad\text{ on } \Sigma_0.
\end{equation}
We did not need the specific form of the extension $\hat{v}^\Sigma$ for the normal part,
but it is required for treating the tangential parts. In fact, with
\[
f^\pm (t,y) = \big[ D_y \Phi^t_{t_0}(y)\big]^{-1}  \, \big( v^+ (t, x)- v^\Sigma (t, x) \big)
\]
for $y\in \Sigma_0$ and $x=\Phi^t_{t_0}(y)\in \Sigma(t)$, condition (\ref{vtan}) implies
\begin{align*}
f^\pm (t,y) & = \big[ D_y \Phi^t_{t_0}(y)\big]^{-1}  \, \langle v^+ (t, x)- v^\Sigma (t, x) , n_\Sigma (t,x) \rangle \, n_\Sigma (t,x) \\[1ex]
& = \langle v^+ (t, x)- v^\Sigma (t, x) , n_\Sigma (t,x) \rangle \, \big[ D_y \Phi^t_{t_0}(y)\big]^{-1}  \,  n_\Sigma (t,x) \\[1ex]
& = \langle v^+ (t, x)- v^\Sigma (t, x) , n_\Sigma (t,x) \rangle \, n_{\Sigma(t_0)} (y)
\end{align*}
by \eqref{normal-evolution}. Consequently, condition (\ref{E14}) becomes
\begin{equation}\label{E25}
f^+_{\|}= f^-_{\|}=0 \quad\text{ on } \Sigma_0.
\end{equation}
\\
{\bf Step 3.} \emph{Reduction to $\Sigma\equiv \R^{n-1} \times \{0\}$}.\\[0.5ex]
By a translation and a rotation, we may assume $x_0=0$ and $n(0)=e_n$, the $n^{\rm th}$ Cartesian base vector. We are only interested in a local result, hence may assume that $\Sigma_0$ is a graph over $\R^{n-1}$ for a height function $h$, i.e.\ \begin{equation}\label{E26}
\Sigma_0=\{x=(x',x_n): x_n=h(x')\}
\end{equation}
with the notation $x'=(x_1,\ldots , x_{n-1})$. Consider the nonlinear transformation
\begin{equation}\label{E27}
x=\begin{bmatrix}
x'\\x_n
\end{bmatrix}
\to H(x)=
\begin{bmatrix}
x'-x_n \nabla_{x'} h(x')/ \big(1+\|\nabla_{x'} h(x')\|^2\big)^{1/2}\\
h(x') + x_n /\big(1+\|\nabla_{x'} h(x')\|^2\big)^{1/2}
\end{bmatrix}.
\end{equation}
For sufficiently small $\varepsilon, r>0$, $H$ is a diffeomorphism from $[\R^{n-1}\times(-\varepsilon,\varepsilon)]\cap B_r(0)$ onto its image $\mathcal N := H ([\R^{n-1}\times (-\varepsilon, \varepsilon)]\cap B_r(0))$,
which is a neighborhood of $0 \in \R^n$.
Given any solution $y(\cdot)$ of (\ref{E22}) starting at $x_0= 0 \in \Sigma_0$, this solution stays inside $\mathcal N$
for $t \in (-\delta , \delta)$,
where $\delta>0$ can be chosen independently of the solution due to the local boundedness of $f$.
The coordinate transformation induced by $H$ yields an a.c.\ function $x(\cdot)$ via
\begin{equation}\label{E28}
y(t)=H (x(t) ),
\end{equation}
which is an a.c.\ solution of
\begin{equation}\label{E29}
\dot x(t)=g\big(t,x(t)\big) \quad\text{ on } J_\delta:=(-\delta , \delta), \quad x(0)=0,
\end{equation}
where $g:J_\delta \times \big( [\R^{n-1}\times(-\varepsilon,\varepsilon)]\cap B_r(0)\big) \to \R^n$ is given as
\begin{equation}\label{gdef}
g(t,x)= \begin{cases}
g^+ (t,x) &\text{ if } x_n \geq 0\\
g^- (t,x) &\text{ if } x_n<0
\end{cases}
\end{equation}
with $g^\pm$ given by
\[
g^\pm (t,x)= H'(x)^{-1}  f^\pm (t, H(x)) \;\;\mbox{ for } x\in \R^n_\pm.
\]
Note that the specific definition of $g$ as $g^+$ for $x_n=0$ in \eqref{gdef} is arbitrary
and the concrete choice of the values there plays no role.
Note also that $g^+:J_\delta \times \R^n_+ \to \R^n$ and $g^- :J_\delta \times \R^n_- \to \R^n$ are jointly continuous and locally Lipschitz continuous in $x$, where $\R^n_\pm$ denote the closed halfspaces $\{x_n\geq 0\}$ and $\{x_n\leq 0\}$, respectively.

Evidently, $y\in \Sigma_0$ iff $x_n=0$ and for such $y=(x',h(x'))$ we have
\begin{equation}\label{n(y)}
n(y)= \frac{1}{(1+\|\nabla_{x'} h(x')\|^2 )^{1/2}}
\begin{bmatrix}
- \nabla_{x'} h(x') \\
1
\end{bmatrix}
\quad \mbox{ for } y=(x',h(x')).
\end{equation}
Given $t\in J_\delta$, $x=(x',0)$ and $y=H(x)\in \Sigma_0$, it holds that
\begin{align*}
& \hat{\rho}^\pm (t,y) f^\pm (t,y)\cdot n(y) \\[1ex]
& = \hat{\rho}^\pm (t,H(x)) \langle H'(x) g^\pm (t,x),
\begin{bmatrix}
- \nabla_{x'} h(x') \\
1
\end{bmatrix}
\rangle (1+\|\nabla_{x'} h(x')\|^2 )^{-1/2}\\[1ex]
& = \hat{\rho}^\pm (t,H(x)) \langle g^\pm (t,x), H'(x)^{\sf T}
\begin{bmatrix}
- \nabla_{x'} h(x') \\
1
\end{bmatrix}
\rangle (1+\|\nabla_{x'} h(x')\|^2 )^{-1/2}.
\end{align*}
Now note that
\begin{equation}\label{H'(x)}
H'(x)=
\left[
\begin{array}{c|c}
 & n_1 (x',h(x')) \\
I_{n-1} & \vdots \\
  & n_{n-1}(x',h(x')) \\[1ex]
\hline \\[-1ex]
 \nabla_{x'} h(x')^{\sf T}   & n_n (x',h(x'))
\end{array}
\right]
\quad \mbox{ for } x=(x',0)
\end{equation}
with $n(x',h(x'))=n(y)$ from \eqref{n(y)}, hence
\[
H'(x)^{\sf T}
\begin{bmatrix}
- \nabla_{x'} h(x') \\
1
\end{bmatrix}
= (1+\|\nabla_{x'} h(x')\|^2 )^{1/2} \, e_n.
\]
Consequently,
\[
\hat{\rho}^\pm (t,y) f^\pm (t,y)\cdot n(y) = \tilde{\rho}^\pm (t,x) \langle g^\pm (t,x), e_n \rangle = \tilde{\rho}^\pm (t,x)\, g_n^\pm (t,x)
\]
with
\begin{equation}\label{tilde-rho}
\tilde{\rho}^\pm (t,x) := \hat{\rho}^\pm (t,H(x)) \quad \mbox{ for } t\in J_\delta,\, x\in [\R^{n-1}\times(-\varepsilon,\varepsilon)]\cap B_r(0).
\end{equation}
This shows that the transmission condition \eqref{E24a} becomes
\begin{equation}\label{E30}
\tilde{\rho}^+(t,x) g^+_n(t,x) = \tilde{\rho}^- (t,x) g^-_n(t,x) \quad\text{ for } t \in J_\delta, \; x_n=0
\end{equation}
with locally Lipschitz continuous $\tilde{\rho}^\pm :J_\delta\times \R^n_\pm \to (0,\infty)$.

Concerning the transformed version of \eqref{E25}, observe that, for $x=(x',0)$,
\[
g^\pm (t,x)=H'(x)^{-1} f^\pm (t,H(x))=H'(x)^{-1} \big(\lambda^\pm (t,x)\, n(H(x)) \big)
\]
with certain $\lambda^\pm (t,x)\in \R$ due to \eqref{E25}.
Hence
\[
g^\pm (t,x)= \lambda^\pm (t,x) H'(x)^{-1}  n(H(x)).
\]
Now note that
\[
H'(x',0) e_n = n(x', h(x'))
\]
by \eqref{H'(x)} with $n(x',h(x'))=n(y)$ from \eqref{n(y)}.
Therefore,
\[
g^\pm (t,x)=  \lambda^\pm (t,x) e_n,
\]
which implies
\begin{equation}\label{E31}
g^\pm_k(t,x)=0 \quad\text{ for } t \in J, \quad x_n=0,\quad k=1,\ldots ,n-1.
\end{equation}
As the result of this step, we may assume that $\Sigma(t) \equiv \R^{n-1} \times \{0\}$ and the new (discontinuous)
right-hand side $g$ has the same regularity as $v$, i.e.\ the $g^\pm$ are continuous on $J\times \R^n_\pm$
and the $g^\pm (t,\cdot)$ are locally Lipschitz continuous on $\R^n_\pm$.
Furthermore, $g$ satisfies the conditions \eqref{E30} and \eqref{E31}.

It remains to show that (\ref{E29}) is uniquely solvable to the right on $[t_0,t_0 + \delta]$ for some $\delta>0$.
\\[1ex]
{\bf Step 4.} \emph{Local forward uniqueness for (\ref{E29})}.\\[0.5ex]
The first argument exploits the physically motivated transmission condition (\ref{transcond}), respectively (\ref{E30}).
Let $x(\cdot)$ and $\overline{x}(\cdot)$ be two solutions of (\ref{E29}).
We then let\\[-2ex]
\begin{equation}\label{E32}
\phi(t)= |\rho(t) x_n(t)-\overline{\rho}(t) \overline{x}_n(t)|
+ || x_{||}(t) -\overline{x}_{||}(t) || \quad\text{ on } J,
\end{equation}
where $||\cdot ||$ denotes the Euclidean norm, $x_{||}=(x_1,\ldots ,x_{n-1},0)$ is the tangential part of $x$,
\begin{equation}\label{E33}
\rho(t)= \begin{cases}
\tilde{\rho}^+ \big(t,x^\Sigma(t)\big) &\text{ if } x_n(t) \geq 0\\
\tilde{\rho}^- \big(t,x^\Sigma(t)\big) &\text{ if } x_n(t)<0
\end{cases}
\end{equation}
with $\tilde{\rho}^\pm$ from \eqref{tilde-rho} and
\begin{equation}\label{E34}
x^\Sigma(t)=
\frac{1}{2}\big(x_{||}(t)+\overline{x}_{||}(t)\big).
\end{equation}
Let $\overline{\rho}(t)$ be defined analogously, exchanging the roles of $x(\cdot)$ and $\overline{x}(\cdot)$.

We first show that $\phi(\cdot)$ is locally Lipschitz continuous, where it suffices to show this for $\rho(\cdot) x_n(\cdot)$, say.
Moreover, for proving \emph{local} Lipschitz continuity, we may assume that
both $g^\pm$ and $\rho^\pm$ are bounded on the respective domain of definition, since they are locally bounded.
At $\tau\in J$ with $x_n(\tau)\not=0$, the local Lipschitz continuity of $\rho(\cdot)x_n(\cdot)$ follows from that of $\rho(\cdot)$ near $\tau$ together with the Lipschitz continuity of $x(\cdot)$.
If $x_n(\tau)=0$, then
\begin{equation}\label{E35}
|\rho(t) x_n(t)- \rho(\tau) x_n(\tau)|= \big|\rho(t) \big(x_n(t)-x_n(\tau)\big)\big| \leq |\rho|_\infty\, |g_n|_\infty\, |t-\tau|,
\end{equation}
hence the Lipschitz estimate holds for $\rho(\cdot) x_n(\cdot)$.

As a consequence, $\phi(\cdot)$ is a.c.\ and a.e.\ differentiable on $J$. Let
\begin{equation}
J_0=\{t\in J: \rho'(t), \overline{\rho}'(t), x'(t), \overline{x}'(t)
\mbox{ exist} \}.
\end{equation}
We are going to show that $\phi'\leq const \; \phi$ a.e.\ on $J$ and it suffices to show this a.e.\ on $J_0$. We distinguish four different cases, where we start by considering $\tau\in J_0$ such that $x_n (\tau)  <0$, $\overline{x}_n (\tau) <0$.
Then $x_n (t)  <0, \, \overline{x}_n (t) <0$ in a neighborhood of $\tau$, hence
$\rho(t)=\overline{\rho}(t)$ there.
This implies
\[
\phi(t)=\rho(t) |x_n(t)-\overline{x}_n(t)|+ || x_{||}(t) -\overline{x}_{||}(t) ||
\quad \mbox{ near } \tau,
\]
hence
\begin{align*}
|\phi'(t)| \; & \leq  \; |\rho'(t)| \, |x_n(t)-\overline{x}_n(t)|
 +\rho(t) |g_n(t,x(t))-g_n(t,\overline{x}(t))|\\[0.5ex]
 & +\; ||g_{||}(t,x(t))-g_{||}(t,\overline{x}(t))||.
\end{align*}
Consequently, using the Lipschitz continuity of $g$ and $||x||\leq |x_n| + ||x_{||}||$,
\begin{align}\label{E39}
|\phi'(t)| \leq \left( \Big|\frac{\rho'(t)}{\rho(t)}\Big| +L \Big(1+\frac{1}{\rho(t)}  \Big)\right) \rho(t) |x_n (t) - \overline{x}_n (t)|\\ \nonumber
+ L (1+\rho(t)) ||x_{||}(t)-\overline{x}_{||}(t)||,
\end{align}
and therefore
\begin{equation}
\label{E40}
\phi' (\tau) \leq K \; \phi (\tau)
\end{equation}
with
\begin{equation}
K:= \max_J \left(\Big|\frac{\rho'}{\rho}\Big| +L \Big(\rho + 1+\frac{1}{\rho}\Big) \right);
\end{equation}
note that $\rho(\cdot)$ is (locally) bounded from below by some $\alpha >0$.

Next, we consider $\tau\in A:=\{t\in J_0: x_n(t)\geq 0, \overline{x}_n(t)\geq 0\}$,
where it suffices to consider those points $\tau$ which are points of Lebesgue density
of $A$. Given such $\tau$, we have
\[
\phi'(\tau)=\lim_{k\to \infty} \frac{\phi(t_k) -\phi(\tau)}{t_k -\tau}
\]
for every sequence $t_k\to \tau$ with $t_k \neq \tau$. Since $\tau$ is a point of Lebesgue density of $A$, we find such a sequence $(t_k)$ in $A$. Then
\[
\phi(t_k)=\rho(t_k) |x_n(t_k)-\overline{x}_n(t_k)|+ || x_{||}(t_k) -\overline{x}_{||}(t_k) ||,
\]
since $\rho(t_k)=\overline{\rho}(t_k)$. Hence $\phi'(t)$ can be estimated
in the same way as above, i.e.\ (\ref{E40}) holds also for such $\tau$.

The remaining two cases can be treated in exactly the same way. We therefore only consider
$\tau\in B:= \{t\in J_0: x_n(t)\geq 0, \overline{x}_n(t)<0 \}$. In fact, it suffices to consider points $\tau$ of Lebesgue density of $B$. In the considered case, we have
\begin{equation}\label{E41}
\phi(t)= \rho(t) x_n(t) - \overline{\rho}(t) \overline{x}_n(t)
 + || x_{||}(t) -\overline{x}_{||}(t) ||\quad\text{ for } t\in B.
\end{equation}
Hence, since $\phi'(\tau)$ exists and can be obtained from difference quotients with $t_k\in B$, we obtain
\begin{equation}\label{E42}
\begin{split}
& \phi'(\tau)= \frac{\rho'(\tau)}{\rho(\tau)} \rho(\tau) x_n(\tau) - \frac{\overline{\rho}'(\tau)}{\overline{\rho}(\tau)} \overline{\rho}(\tau) \overline{x}_n(\tau)\\
&+ \rho(\tau) g^+_n\big(\tau,x(\tau)\big)- \overline{\rho} (\tau) g^-_n\big(\tau,\overline{x}(\tau)\big)
+\frac{d}{dt} || x_{||}(\tau) -\overline{x}_{||}(\tau) ||.
\end{split}
\end{equation}
By means of (\ref{E30}), we have
\begin{align*}
& \rho(\tau) g^+_n\big(\tau,x(\tau)\big)- \overline{\rho} (\tau) g^-_n\big(\tau,\overline{x}(\tau)\big) =\\[0.5ex]
& \rho(\tau) \big( g^+_n\big(\tau,x(\tau)\big)-g^+_n\big(\tau,x^\Sigma(\tau)\big) \big)
 + \overline{\rho} (\tau) \big( g^-_n\big(\tau,\overline{x}(\tau)\big) - g^-_n\big(\tau,x^\Sigma (\tau)\big) \big).
\end{align*}
Using the Lipschitz continuity of $g$ as well as $||x(t)-x^\Sigma (t)||\leq ||x(t)-\overline{x}(t)||$ and the corresponding inequality for $||\overline{x}(t)-x^\Sigma (t)||$, equation (\ref{E42}) implies
\begin{align*}\label{E43}
|\phi'(\tau)| \leq \left(\Big|\frac{\rho'(\tau)}{\rho(\tau)}\Big| +\Big|\frac{\overline{\rho}'(\tau)}{\overline{\rho}(\tau)}\Big|\right)
|\rho(\tau) x_n(\tau) - \overline{\rho}(\tau) \overline{x}_n(\tau)|\\[0.5ex]
+ L ( 1+ 2 |\rho|_\infty ) || x(\tau) -\overline{x}(\tau) ||;
\end{align*}
recall that $-\overline{x}_n(\tau) >0$.
Splitting $x(\tau)$ and $\overline{x}(\tau)$ into their normal and tangential parts,
this yields (\ref{E40}) with
\begin{equation}
K:= \max_J \left( \Big|\frac{\rho'}{\rho}\Big|
+\Big|\frac{\overline{\rho}'}{\overline{\rho}}\Big|
+ L \Big( 1+ |\rho|_\infty + |\overline{\rho}|_\infty \Big)
\Big(1+ \frac{1}{\rho} + \frac{1}{\overline{\rho}} \Big) \right).
\end{equation}
Recall that both $\rho(\cdot)$, 
$\overline{\rho}(\cdot)$
are (locally) bounded from below by some $\alpha >0$.

Consequently, inequality (\ref{E40}) holds a.e.\ on $J$ with a common $K>0$,
thus $\phi(t)= 0$ on $J$ by Gronwall's Lemma, since $\phi(0)=0$. This means
\begin{equation}
\label{phi-result}
\rho(t) x_n(t) = \overline{\rho}(t) \overline{x}_n(t) \quad \mbox{ and } \quad
x_{||}(t)=\overline{x}_{||}(t) \quad \mbox{ on } J.
\end{equation}
To finish the proof, consider the energy functional
\[
\psi (t)=\frac{1}{2} ||x(t)- \overline{x}(t)||^2 \quad \mbox{ for } t\in J.
\]
Evidently, using (\ref{phi-result})$_2$, we get
\begin{equation}
\label{psi}
\psi' (t)=\big(x_n(t) - \overline{x}_n(t) \big) \,
 \big( g_n (t,x(t)) - g_n(t, \overline{x}(t)) \big) \quad \mbox{ on } J.
\end{equation}
By (\ref{phi-result})$_1$ and the non-degeneracy of $\rho(\cdot)$ and $\overline{\rho}(\cdot)$, both $x_n(\cdot)$ and $\overline{x}_n(\cdot)$ run either in $\R_+^n$ or in $\R_-^n$.
Hence the second argument of $g_n$  in (\ref{psi}) is always in the same halfspace,
i.e.\ $g_n$ is either $g_n^+$ or $g_n^-$.
By the Lipschitz continuity of $g^\pm$ on $J\times \R_\pm^n$, we obtain $\psi'(t)\leq 2 L \psi(t)$ on $J$. Hence $\psi(t)=0$ on $J$, i.e.\ $x(t)=\overline{x}(t)$ on $J$ which ends the proof.
\\[1ex]
{\bf Step 5.} \emph{Local backward uniqueness}.\\[0.5ex]
The vector field $\tilde v:=-v$ satisfies all  assumptions of Theorem~\ref{T1} if $v^\Sigma$ is replaced by $-v^\Sigma$.
Hence \eqref{E12} with $-v$ instead of $v$ and the backward moving $\Sigma (\cdot)$ has unique local forward solvability.
Reversing time, this yield unique local backward solvability of the original problem.
\\
$\mbox{ } \hfill \Box$\\[1ex]
Let us finally remark that in wetting applications, the fluid interface has contact with parts of the boundary of $\Omega$,
typically at a solid wall. This leads to a technically more involved case with moving contact lines
for which the present result is a helpful and necessary starting point.\\[2ex]
{\bf Acknowledgment.}
The author gratefully acknowledges financial support by the German Research Foundation (DFG) within the Collaborative Research Center 1194 \emph{Interaction of Transport and Wetting Processes}, project B01.

Over the last two decades, the author frequently had the great pleasure to discuss about two-phase flows with the late Jan Pr\"u{\ss}.
This has added a lot to my understanding of this topic. Danke, Jan!


\begin{thebibliography}{99}
\baselineskip11pt
{\small
\bibitem{AmbCrippa} L.\ Ambrosio, G.\ Crippa: Continuity equations and ODE flows with non-smooth velocity.
Proceedings of the Royal Society of Edinburgh {\bf 144A}, 1191-1244 (2014).
\bibitem{Ac} J.P.\ Aubin, A.\ Cellina: {\it Differential Inclusions},
Springer 1984.
\bibitem{AubinFrankowska} J.P.\ Aubin, H.\ Frankowska: {\it Set-valued Analysis},
Birkh\"auser 1990.
\bibitem{Bo2} D.\ Bothe: Multivalued differential equations on graphs.
Nonlinear Analysis {\bf 18}, 245-252 (1992).
\bibitem{Bo15} D. Bothe: The instantaneous limit of a reaction-diffusion
system, pp.\ 215-224 in {\it Evolution Equations and Their Applications in Physical and Life Sciences} (G. Lumer, L. Weis, eds).
Lect.\ Notes Pure Appl.\ Math.\ {\bf 215}. Marcel Dekker 2000.
\bibitem{BPS-surfactant} D.\ Bothe, J.\ Pr\"uss, G.\ Simonett: Well-posedness of a two-phase
flow with soluble surfactant, pp.\ 37-61 in Nonlinear Elliptic
and Parabolic Problems (M. Chipot, J. Escher, eds.), Birkh\"auser 2005.
\bibitem{BoWi} D.\ Bothe, P.\ Wittbold: Abstract reaction-diffusion systems with m-completely accretive diffusion operators
and measurable reaction rates. Commun.\ Pure Appl.\ Anal.\ {\bf 11} (6), 2239-2260 (2012).
\bibitem{Bressan88} A.\ Bressan: Directionally continuous selections and differential inclusions.
Funkc.\ Ekvac.\ {\bf 31}, 459-470 (1988).
\bibitem{Bressan-uniqueness} A.\ Bressan:
Unique Solutions for a Class of Discontinuous Differential Equations.
Proceedings of the American Mathematical Society {\bf 104} (3), 772-778 (1988).
\bibitem{Cortes} J.\ {Cortes}: Discontinuous dynamical systems.
IEEE Control Systems Magazine {\bf 28} (3), 36-73 (2008).
\bibitem{MDE} K.\ Deimling: {\it Multivalued Differential Equations.}
De Gruyter 1992.
\bibitem{DiPLions} R.\ J.\ DiPerna P.\ L.\ Lions: Ordinary differential equations, transport theory and Sobolev spaces.
Inventiones mathematicae {\bf 98}, 511�547 (1989).
\bibitem{Fil} A.F.\ Filippov: {\it Differential Equations with
Discontinuous Right-Hand Sides}. Kluwer 1988.
\bibitem{FKB-2019} M.\ Fricke, M.\ K\"ohne, D.\ Bothe:
A kinematic evolution equation for the dynamic contact angle and some consequences.
Physica D: Nonlinear Phenomena {\bf 394}, 26-43 (2019).
\bibitem{Georgescu} C.\ Georgescu, B.\ Brogliato, V.\ Acary:
Switching, relay and complementarity systems: A tutorial on their well-posedness and relationships. Physica D: Nonlinear Phenomena
{\bf 241} (22), 1985-2002 (2012).
\bibitem{Giga.2006}
Y.\ Giga: {\it Surface evolution equations: A level set approach.}
Monographs in Mathematics {\bf 99}. Birkh{\"a}user, Basel (2006).
\bibitem{Hartman}
P.\ Hartman: {\it Ordinary Differential Equations.} Second Edition, Birkh\"auser, 1982.
\bibitem{Kimura.2008}
M.\ Kimura: Geometry of hypersurfaces and moving hypersurfaces in $\R^m$ - For the study of moving boundary problems,
pp.\ 39-93 in Lecture notes Volume IV, Topics in Mathematical Modeling, Jind\v{r}ich Ne\v{c}as Center for Mathematical Modeling
(M.\ Benes, E.\ Feireisl, eds). Matfyzpress 2008.
\bibitem{Peano}
G.\ Peano: D{\'e}monstration de l'int{\'e}grabilit{\'e} des {\'e}quations diff{\'e}rentielles ordinaires.
Mathematische Annalen {\bf 37} (2), 182-228 (1890).
\bibitem{PrSi15} J.~Pr\"uss, G.~Simonett: {\em Moving Interfaces and Quasilinear Parabolic Problems.}
Monographs in Mathematics, Birkh\"auser 2016.
\bibitem{PW-ODE} J.~Pr\"uss, M.\ Wilke: {\em Gew\"ohnliche Differentialgleichungen und dynamische Systeme.}
Grundstudium Mathematik, Birkh\"auser 2010.
\bibitem{Slattery-Interfaces}
J.C.\ Slattery, L.\ Sagis, E.-S.\ Oh: {\it Interfacial Transport Phenomena (2$^{\rm nd}$ ed.)}. Springer, New York 2007.
\bibitem{Walter-Ana2}
W.\ Walter: {\it Analysis II}. Grundwissen Mathematik 4, Springer 1990.
}
\end{thebibliography}
\end{document}